\documentclass{amsart}

\usepackage{amsfonts,amssymb,amscd,amsmath,latexsym,amsbsy,stmaryrd}
\usepackage[hypertex]{hyperref}
\usepackage{amsrefs}
\usepackage[matrix,arrow,tips]{xy}
\usepackage{mathrsfs}
\usepackage{amsthm}
\usepackage{tikz}
\usetikzlibrary{decorations}
\usetikzlibrary{arrows}
\usetikzlibrary{shapes}
\usepackage[active]{srcltx}

\newcommand{\bbC}{\mathbb{C}}
\newcommand{\bbZ}{\mathbb{Z}}
\newcommand{\bbN}{\mathbb{N}}
\newcommand{\bbQ}{\mathbb{Q}}
\newcommand{\bbR}{\mathbb{R}}

\newcommand{\scrH}{\mathscr{H}}

\newcommand{\tr}{\mathrm{tr}}
\newcommand{\End}{\mathrm{End}}

\newcommand{\diag}{\mathrm{diag}}
\newcommand{\Mat}{\mathrm{Mat}}

\newcommand{\laa}{\mathfrak{a}}
\newcommand{\lag}{\mathfrak{g}}

\newcommand{\lak}{\mathfrak{k}}

\newcommand{\SU}{\mathrm{SU}}

\newcommand{\GL}{\mathrm{GL}}
\newcommand{\U}{\mathrm{U}}

\newcommand{\rFs}[5]{\,_{#1}F_{#2} \left( \genfrac{.}{.}{0pt}{}{#3}{#4}
\ ;#5 \right)}

\newtheorem{theorem}{Theorem}[section]
\newtheorem{definition}[theorem]{Definition}
\newtheorem{lemma}[theorem]{Lemma}
\newtheorem{cor}[theorem]{Corollary}
\newtheorem{prop}[theorem]{Proposition}

\newtheorem{conjecture}[theorem]{Conjecture}

\theoremstyle{definition}
\newtheorem{remark}{Remark}[section]

\numberwithin{equation}{section}

\title[Matrix Valued Orthogonal Polynomials]{Matrix Valued Orthogonal Polynomials\\related to $(\SU(2)\times\SU(2),\diag)$}

\author{Erik Koelink, Maarten van Pruijssen, Pablo Rom\'an}

\address{Radboud Universiteit, IMAPP\\
Heyendaalseweg 135\\
6525 GL Nijmegen\\
The Netherlands}
\email{e.koelink@math.ru.nl, m.vanpruijssen@math.ru.nl}
\address{Department of Mathematics\\
Katholieke Universiteit Leuven\\
Celestijnenlaan 200B, B-3001\\
Leuven, Belgium.}
\email{pablo.roman@wis.kuleuven.be}

\begin{document}

\date{\today}

\begin{abstract}
The matrix-valued spherical functions for the pair $(K\times K, K)$, $K=\mathrm{SU(2)}$, are
studied.
By restriction to the subgroup $A$ the matrix-valued spherical functions
are diagonal. 
For suitable set of representations we take these diagonals into a matrix-valued
function,
which are the full spherical functions.
Their orthogonality is a consequence of the Schur orthogonality relations.
From the full spherical functions we obtain  matrix-valued orthogonal polynomials of
arbitrary size, and 
they satisfy a three-term recurrence relation which follows by considering tensor
product decompositions.
An explicit expression for the weight and the complete block-diagonalization of the
matrix-valued orthogonal polynomials is obtained.
From the explicit expression we obtain right-hand sided differential operators of
first and 
second order for which
the matrix-valued orthogonal polynomials are eigenfunctions. 
We study the low-dimensional cases explicitly, and for these cases additional results,
such as the Rodrigues' formula and being eigenfunctions to first order
differential-difference
and second order differential operators, are obtained.
\end{abstract}
\maketitle


\section{Introduction}\label{1}

The connection between special functions and representation theory of Lie groups is 
a very fruitful one, see e.g. \cite{Vilenkin}, \cite{VilenkinKlimyk3vol}. For the 
special case of the group $\SU(2)$ we know that the matrix elements of the irreducible
finite-dimensional representations are explicitly expressible in terms of Jacobi
polynomials,
and in this way many of the properties of the Jacobi polynomials can be obtained from 
the group theoretic interpretation. In particular, the spherical functions with respect
to the subgroup $\mathrm{S}(\U(1)\times \U(1))$ are the Legendre polynomials, and using this
interpretation one obtains product formula, addition formula, integral formula, etc.
for the Legendre polynomials, see e.g. \cite{VaradarajanGangolli}, \cite{Heckman},
\cite{HelgasonGGA}, \cite{Vilenkin}, \cite{VilenkinKlimyk3vol}
for more information on spherical functions. 

In the development of spherical functions for a symmetric pair $(G,K)$ the emphasis
has been 
on spherical functions with respect to  one-dimensional representations of $K$, and
in particular 
the trivial representation  of $K$. 
Godement \cite{Godement} considered the case of higher-dimensional representations
of $K$, 
see also  \cite{VaradarajanGangolli}, \cite{TiraoSF} for the general theory.
Examples studied
are \cite{Camporesi2000}, \cite{vanDijk}, \cite{GPT}, \cite{Koornwinder85}, \cite{Pedon}.
However, the focus is 
usually not on obtaining explicit expressions for the matrix-valued spherical
functions, see 
Section \ref{2} for the definition, except for \cite{GPT} and \cite{Koornwinder85}.
In \cite{GPT} the matrix-valued
spherical functions are studied for the case $(U,K)=(\SU(3),\U(2))$, and the
calculations
revolve around the study of the algebra of differential operators for which these 
matrix-valued orthogonal polynomials are eigenfunctions. The approach in this paper is 
different. 

In our case the paper \cite{Koornwinder85} by Koornwinder is relevant. Koornwinder
studies
the case of the compact symmetric pair $(U,K)=(\SU(2)\times \SU(2), \SU(2))$ where the
subgroup is
diagonally embedded, and he calculates explicitly vector-valued orthogonal
polynomials. The 
goal of this paper is to study this example in more detail and to study the
matrix-valued
orthogonal polynomials arising from this example. The spherical functions in this case
are the characters of $\SU(2)$, which are the Chebychev polynomials of the second kind
corresponding to the Weyl character formula. 
So the matrix-valued orthogonal polynomials can be considered as analogues of the
Chebychev polynomials. Koornwinder \cite{Koornwinder85} introduces the vector-valued
orthogonal polynomials which coincide with rows in the matrix of the matrix-valued
orthogonal polynomials in this paper. We provide some of Koornwinder's results with
new proofs.
The matrix-valued spherical functions can be given explicitly in terms of the
Clebsch-Gordan
coefficients, or $3-j$-symbols, of $\SU(2)$. 
Moreover, we find many more properties of these matrix-valued orthogonal polynomials.
In particular, we give an explicit expression for the weight, i.e. the matrix-valued
orthogonality measure, in terms of Chebychev polynomials by using an expansion in terms
of spherical functions of the matrix elements and explicit knowledge of Clebsch-Gordan
coefficient. This gives some strange identities for 
sums of hypergeometric functions in Appendix  \ref{appendixA}. Another important result is
the 
explicit three-term recurrence relation which is obtained by considering tensor 
product decompositions. Also, using the explicit expression for the weight function we
can obtain differential operators for which these matrix-valued orthogonal polynomials
are eigenfunctions. 

Matrix-valued orthogonal polynomials arose in the work of Krein \cite{Krein1},
\cite{Krein2} and have
been studied from an analytic point of view by Dur\'an and others, see 
\cite{Duran1}, \cite{DG}, \cite{DG2}, \cite{DG3}, \cite{DuranLR} and references
given there.
As far as we know, the matrix-valued
orthogonal polynomials that we obtain have not been considered before. Also $2\times
2$-matrix-valued orthogonal
polynomials occur in the approach of the non-commutative oscillator, see \cite{IW}
for more
references. A group theoretic interpretation of this oscillator in general seems to
be lacking. 

The results of this paper can be generalized in various ways. 
First of all, the approach can be generalized to pairs $(U,K)$ with $(U\lag)^{\lak}$
abelian, 
but this is rather restrictive \cite{knop}.
Given a pair $(U,K)$ and a representation $\delta$ of $K$ such that
$[\pi|K:\delta]\le1$ for all representations $\pi$ of $G$ and $\delta|_{M}$ is
multiplicity free, we can perform the same construction to get matrix-valued
orthogonal polynomials.
Needless to say, in general it might be difficult to be able to give an explicit
expression of the weight function. 
Another option is to generalize to $(K\times K, K)$ to obtain matrix-valued
orthogonal polynomials generalizing
Weyl's character formula for other root systems, see e.g. \cite{Heckman}. 

We now discuss the contents of the paper. In Section \ref{2} we introduce the
matrix-valued spherical functions
for this pair taking values in the matrices of size $(2\ell+1)\times (2\ell+1)$,
$\ell\in\frac12 \bbN$. In Section \ref{3} we prove the recurrence relation for the
matrix-valued spherical functions
using a tensor product decomposition. This result gives us the opportunity to
introduce polynomials, and
this coincides with results of Koornwinder \cite{Koornwinder85}. In Section \ref{4}
we introduce the 
full spherical functions on the subgroup $A$, corresponding to the Cartan
decomposition  $U=KAK$, 
by putting the restriction to $A$ of the matrix-valued spherical function into a
suitable matrix. 
In Section \ref{5} we discuss the explicit form and the symmetries of the weight.
Moreover, we calculate
the commutant explicitly and this gives rise to a decomposition of the full
spherical functions, 
the matrix-valued orthogonal polynomials and the weight function in a $2\times
2$-block diagonal matrix,
which cannot be reduced further. After a brief review of generalities of
matrix-valued orthogonal
polynomials in Section \ref{6}, we discuss the even and odd-dimensional cases
separately. In the even
dimensional case an interesting relation between the two blocks occur. In Section
\ref{7} we discuss
the right hand sided differential operators, and we show that the matrix-valued
orthogonal polynomials
associated to the full spherical function are eigenfunctions to a first order
differential
operator as well as to a second order differential operator.  Section \ref{8}
discusses explicit
low-dimensional examples, and gives some additional information such as the Rodrigues'
formula for these matrix-valued orthogonal polynomials and more differential
operators. Finally, in the appendices we give
somewhat more technical proofs of two results. 

\textbf{Acknowledgement.} We thank Mizan Rahman 
and Gert Heckman for useful discussions. 
Pablo Rom\'an is supported by Katholieke Universiteit Leuven research
grant OT/08/33 and by Belgian Interuniversity Attraction Pole P06/02.


\section{Spherical Functions of the pair $(\SU(2)\times\SU(2),\diag)$}\label{2}

Let $K=\SU(2)$, $U=K\times K$ and $K_{*}\subset U$ the diagonal
subgroup. An element in $K$ is of the form
\begin{equation}
k(\alpha,\beta)=\left(\begin{array}{cc}\alpha&\beta\\
-\overline{\beta}&\overline{\alpha}\end{array}\right),\quad|\alpha|^{2}+|\beta|^{2}=1,\quad \alpha,\beta\in\bbC.
\end{equation}
Let $m_{t}:=k(e^{it/2},0)$ and let $T\subset K$ be the
subgroup consisting of the $m_{t}$. $T$ is the (standard) maximal
torus of $K$. The subgroup $T\times T\subset U$ is a maximal torus of
$U$. Define
$$A_{*}=\{(m_{t},m_{-t}):0\le t<4\pi\}\quad\mbox{and}\quad
M=\{(m_{t},m_{t}):0\le t<4\pi\}.$$
We write $a_{t}=(m_{t},m_{-t})$ and
$b_{t}=(m_{t},m_{t})$. We have $M=Z_{K_{*}}(A_{*})$ and the decomposition
$U=K_{*}A_{*}K_{*}$. Note that $M$ is
the standard maximal torus of $K_{*}$.

The equivalence classes of the unitary irreducible representations of
$K$ are paramatrized by $\widehat{K}=\frac{1}{2}\bbN$. An element
$\ell\in\frac{1}{2}\bbN$ determines the space
$$H^{\ell}:=\bbC[x,y]_{2\ell},$$
the space of homogeneous polynomials of degree $2\ell$ in the variables
$x$ and $y$. We view this space as a subspace of the function space
$C(\bbC^{2},\bbC)$ and
as such, $K$ acts naturally on it via
$$k:p\mapsto p\circ k^{t},$$
where $k^{t}$ is the transposed. Let
\begin{equation}\label{standard weight basis}
\psi^{\ell}_{j}:(x,y)\mapsto\binom{2\ell}{\ell-j}^{\frac{1}{2}}x^{\ell-j}y^{\ell+j},\quad j=-\ell,-\ell+1,\ldots,\ell-1,\ell
\end{equation}
We stipulate that this is an
orthonormal basis with respect to a Hermitian inner product that is
linear in the first variable. The representation $T^{\ell}:K\to\GL(H^{\ell})$ is irreducible and unitary.

The equivalence classes of the unitary irreducible representations of
$U$ are paramatrized by
$\widehat{U}=\widehat{K}\times\widehat{K}=\frac{1}{2}\bbN\times\frac{1}{2}\bbN$.
An element $(\ell_{1},\ell_{2})\in\frac{1}{2}\bbN\times\frac{1}{2}\bbN$
gives rise to the Hilbert space
$H^{\ell_{1},\ell_{2}}:=H^{\ell_{1}}\otimes H^{\ell_{2}}$ and in turn to
the irreducible unitary representation on this space, given by the outer tensor product
$$T^{\ell_{1},\ell_{2}}(k_{1},k_{2})(\psi^{\ell_{1}}_{j_{1}}\otimes\psi^{\ell_{2}}_{j_{2}})
=T^{\ell_{1}}(k_{1})(\psi^{\ell_{1}}_{j_{1}})\otimes
T^{\ell_{2}}(k_{2})(\psi^{\ell_{2}}_{j_{2}}).$$

The restriction of $(T^{\ell_{1},\ell_{2}},H^{\ell_{1},\ell_{2}})$ to
$K_{*}$ decomposes multiplicity free in summands of type
$\ell\in\frac{1}{2}\bbN$ with
\begin{equation}\label{eqn:parametrization}
|\ell_{1}-\ell_{2}|\le\ell\le\ell_{1}+\ell_{2}\quad\mbox{and}\quad\ell_{1}+\ell_{2}-\ell\in\bbZ.
\end{equation}
Conversely, the representations of $U$ that contain a given
$\ell\in\frac{1}{2}\bbN$ are the pairs
$(\ell_{1},\ell_{2})\in\frac{1}{2}\bbN\times\frac{1}{2}\bbN$ that
satisfy (\ref{eqn:parametrization}). We have pictured this parametrization
in Figure \ref{fig:1stclassif} for $\ell=3/2$.

\begin{figure}
\begin{center}
\resizebox{.424\textwidth}{!}{
\begin{tikzpicture}[scale=0.5,>=stealth]
\draw[->] (0,0) -- (9,0) node[below] {$\ell_1$};
\draw[->] (0,0) -- (0,9) node[left] {$\ell_2$};

\draw[very thin,dotted] (0,1) -- (9,1);
\draw[very thin,dotted] (0,2) -- (9,2);
\draw[very thin,dotted] (0,3) -- (9,3);
\draw[very thin,dotted] (0,4) -- (9,4);
\draw[very thin,dotted] (0,5) -- (9,5);
\draw[very thin,dotted] (0,6) -- (9,6);
\draw[very thin,dotted] (0,7) -- (9,7);
\draw[very thin,dotted] (0,8) -- (9,8);

\draw[very thin,dotted] (1,0) -- (1,9);
\draw[very thin,dotted] (2,0) -- (2,9);
\draw[very thin,dotted] (3,0) -- (3,9);
\draw[very thin,dotted] (4,0) -- (4,9);
\draw[very thin,dotted] (5,0) -- (5,9);
\draw[very thin,dotted] (6,0) -- (6,9);
\draw[very thin,dotted] (7,0) -- (7,9);
\draw[very thin,dotted] (8,0) -- (8,9);

\draw[fill=black] (0,1) circle (2pt) node[left]{$\frac12$};
\draw[fill=black] (0,2) circle (2pt) node[left]{$2$};
\draw[fill=black] (0,3) circle (2pt) node[left]{$\frac32$};
\draw[fill=black] (0,4) circle (2pt) node[left]{$3$};
\draw[fill=black] (0,5) circle (2pt) node[left]{$\frac52$};
\draw[fill=black] (0,6) circle (2pt) node[left]{$4$};
\draw[fill=black] (0,7) circle (2pt) node[left]{$\frac72$};
\draw[fill=black] (0,8) circle (2pt) ;

\draw[fill=black] (1,0) circle (2pt) node[below]{$\frac12$};;
\draw[fill=black] (2,0) circle (2pt) node[below]{$2$};;
\draw[fill=black] (3,0) circle (2pt) node[below]{$\frac32$};;
\draw[fill=black] (4,0) circle (2pt) node[below]{$3$};
\draw[fill=black] (5,0) circle (2pt) node[below]{$\frac52$};
\draw[fill=black] (6,0) circle (2pt) node[below]{$4$};
\draw[fill=black] (7,0) circle (2pt) node[below]{$\frac72$};
\draw[fill=black] (8,0) circle (2pt);
\draw[fill=black] (0,0) circle (2pt) node[below] {$0$};

\draw[] (0,3) -- (3,0);

\draw[] (0,3) -- (6,9);
\draw[] (3,0) -- (9,6);

\draw[] (1,2) -- (8,9);
\draw[] (2,1) -- (9,8);

\draw[dotted] (6,9) -- (7,10);
\draw[dotted] (9,6) -- (10,7);

\draw[] (1,4) -- (4,1);
\draw[] (2,5) -- (5,2);
\draw[] (3,6) -- (6,3);
\draw[] (4,7) -- (7,4);
\draw[] (5,8) -- (8,5);
\draw[] (6,9) -- (9,6);



\end{tikzpicture}}
\end{center}
\caption{Plot of the parametrization of the pairs $(\ell_{1},\ell_{2})$ that contain $\ell$ upon restriction.}\label{fig:1stclassif}
\end{figure}
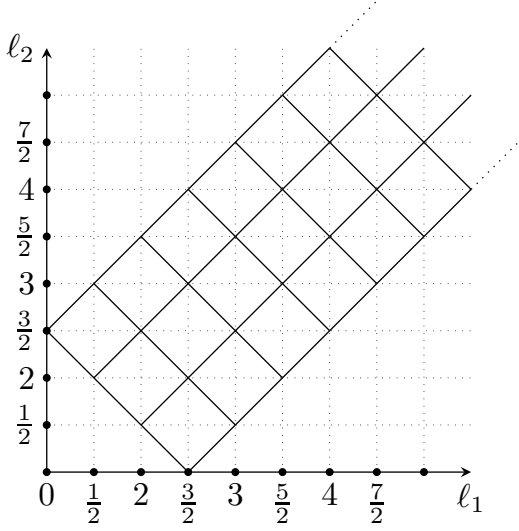

The following theorem is standard, see \cite{Koornwinder81}.
\begin{theorem}\label{thm:cgbasis}
The space $H^{\ell_{1},\ell_{2}}$ has a basis
$$\{\phi^{\ell_{1},\ell_{2}}_{\ell,j}:\ell\mbox{ satisfies
(\ref{eqn:parametrization}) and }|j|\le\ell\}$$
such that for every $\ell$ that the map
$\beta_{\ell}^{\ell_{1},\ell_{2}}:H^{\ell}\to H^{\ell_{1},\ell_{2}}$
defined by $\psi^{\ell}_{j}\mapsto\phi^{\ell_{1},\ell_{2}}_{\ell,j}$ is a
$K$-intertwiner. The base change with respect to the standard basis
$\{\psi^{\ell_{1}}_{j_{1}}\otimes\psi^{\ell_{2}}_{j_{2}}\}$ of
$H^{\ell_{1},\ell_{2}}$ is given by
$$\phi^{\ell_{1},\ell_{2}}_{\ell,j}=\sum_{j_{1}=-\ell_{1}}^{\ell_{1}}\sum_{j_{2}
=-\ell_{2}}^{\ell_{2}}C^{\ell_{1},\ell_{2},\ell}_{j_{1},j_{2},j} \,
\psi^{\ell_{1}}_{j_{1}}\otimes\psi^{\ell_{2}}_{j_{2}},$$
where the $C^{\ell_{1},\ell_{2},\ell}_{j_{1},j_{2},j}$ are the Clebsch-Gordan coefficients, normalized in the standard way. The Clebsch-Gordan coefficient satisfies $C^{\ell_{1},\ell_{2},\ell}_{j_{1},j_{2},j}=0$ if $j_{1}+j_{2}\ne j$.
\end{theorem}

\begin{definition}[Spherical Function]\label{def:sf} Fix a $K$-type $\ell\in\frac{1}{2}\bbN$ and let $(\ell_{1},\ell_{2})\in\frac{1}{2}\bbN\times\frac{1}{2}\bbN$ be a
representation that contains $\ell$ upon restriction to $K_{*}$. The
spherical function of type $\ell\in\frac{1}{2}\bbN$ associated to
$(\ell_{1},\ell_{2})\in\frac{1}{2}\bbN\times\frac{1}{2}\bbN$ is defined by
\begin{equation}\Phi_{\ell_{1},\ell_{2}}^{\ell}:U\to\End(H^{\ell}):x\mapsto\left(\beta^{\ell_{1},
\ell_{2}}_{\ell}\right)^{*}\circ
T^{\ell_{1},\ell_{2}}(x)\circ\beta^{\ell_{1},\ell_{2}}_{\ell}.
\end{equation}
\end{definition}
If $\Phi_{\ell_{1},\ell_{2}}^{\ell}$ is a spherical function of type
$\ell$ then it satisfies the following properties:
\begin{enumerate}\setlength{\itemsep}{1.5mm} 
\item $\Phi_{\ell_{1},\ell_{2}}^{\ell}(e)=I$, where $e$ is the identity element in the group $U$ and $I$ is the
identity transformation in $H^\ell$,
\item $\Phi_{\ell_{1},\ell_{2}}^{\ell}(k_{1}xk_{2})=T^{\ell}(k_{1})\Phi_{\ell}(x)T^{\ell}(k_{2})$ 
for all $k_{1},k_{2}\in K_{*}$ and $x\in U$,
\item $\Phi_{\ell_{1},\ell_{2}}^{\ell}(x)\Phi_{\ell_{1},\ell_{2}}^{\ell}(y)
=\int_{K^*}\chi_\ell(k^{-1})\Phi_{\ell_{1},\ell_{2}}^{\ell}(xky)dk$, for all $x,y\in U$. Here 
$\xi_\ell$ denotes the character of $T^\ell$ and $\chi_\ell=(2\ell+1)\xi_\ell$.
\end{enumerate}

\begin{remark}
Definition \ref{def:sf} is not the definition of a spherical function given by
Godement \cite{Godement}, Gangolli and Varadarajan \cite{VaradarajanGangolli} or Tirao \cite{TiraoSF} but it follows from property $(3)$ that it is equivalent in this situation.
The point where our definition differs is essentially that we choose one
space, namely $\End(H^{\ell})$, in which all the spherical functions take their values,
instead of different endomorphism rings for every $U$-representation. We can do this because
of the multiplicity free splitting of the irreducible representations.
\end{remark}

\begin{prop}\label{prop:diag}
Let $\End_{M}(H^{\ell})$ be the the algebra of elements
$Y\in\End(H^{\ell})$ such that $T^{\ell}(m)Y=YT^{\ell}(m)$ for all $m\in
M$. Then
$\Phi_{\ell_{1},\ell_{2}}^{\ell}(A_{*})\subset\End_{M}(H^{\ell})$. The restriction of $\Phi_{\ell_{1},\ell_{2}}^{\ell}$ to $A_{*}$ is diagonalizable.
\end{prop}
\begin{proof}
This is observation \cite{Koornwinder85}*{(2.6)}. Another
proof, similar to \cite{GPT}*{Prop. 5.11}, uses $ma=am$ for all $a\in A_{*}$ and $m\in M$ so that by $(2)$
$$\Phi_{\ell_{1},\ell_{2}}^{\ell}(a)=T^{\ell}(m)\Phi_{\ell_{1},\ell_{2}}^{\ell}(a)T^{\ell}(m)^{-1}.$$
The second statement follows from the fact that the restriction of any irreducible representation of
$K_{*}\cong\SU(2)$ to $M\cong\U(1)$ decomposes multiplicity free.
\end{proof}
The standard weight basis (\ref{standard weight basis}) is a weight basis in which $\Phi_{\ell_{1}\ell_{2}}^{\ell}|_{A_{*}}$ is diagonal. The restricted spherical functions are given by
\begin{equation}\label{eqn:ressfGENERAL}
\left(\Phi_{\ell_{1},\ell_{2}}^{\ell}(a_{t})\right)_{j,j}=\sum_{j_{1}
=-\ell_{1}}^{\ell_{1}}\sum_{j_{2}=-\ell_{2}}^{\ell_{2}}e^{i(j_{2}-j_{1})t}
\left(C^{\ell_{1},\ell_{2},\ell}_{j_{1},j_{2},j}\right)^{2},
\end{equation}
which follows from Definition \ref{def:sf} and Theorem \ref{thm:cgbasis}. 


\section{Recurrence Relation for the Spherical Functions}\label{sec:recurrence_relation}\label{3}

A zonal spherical function is a spherical function $\Phi_{\ell_{1},\ell_{2}}^{\ell}$ for the trivial
$K$-type $\ell=0$. We have a diffeomorphism $U/K_{*}\to K:(k_{1},k_{2})K_{*}\mapsto k_{1}k_{2}^{-1}$ and the left $K_{*}$-action on $U/K_{*}$ corresponds to the action of $K$ on itself by conjugation. The zonal spherical functions are the characters on $K$ \cite{Vilenkin} which are parametrized by pairs
$(\ell_{1},\ell_{2})$ with $\ell_{1}=\ell_{2}$ and we write $\varphi_{\ell}=\Phi_{\ell,\ell}^{0}$. Note that $\varphi_{\ell}=(-1)^{-j+l}(2\ell+1)^{-1/2}U_{2\ell}(\cos t)$ by (\ref{eqn:ressfGENERAL}) and $C^{\ell,\ell,0}_{j,-j,0}=(-1)^{-j+l}(2\ell+1)^{-1/2}$, where $U_{n}$ is the Chebyshev polynomial of the second kind of degree $n$. The zonal spherical function $\varphi_{\frac12}$ plays an
important role and we denote it  by $\varphi=\varphi_{\frac12}$.
Any other zonal spherical function
$\varphi_{n}$ can be expressed as a polynomial in $\varphi$, see e.g.
\cite{Vilenkin}, \cite{Vretare}. For the spherical functions we
obtain a similar result. Namely, the product of $\varphi$ and a spherical function
of type $\ell$ can be written as a linear combination of
at most four spherical functions of type $\ell$.
\begin{prop}\label{prop:recursion} We have as functions on $U$
\begin{equation}
\varphi\cdot\Phi^{\ell_{1},\ell_{2}}_{\ell}=\sum_{m_{1}=|\ell_{1}-
\frac{1}{2}|}^{\ell_{1}+\frac{1}{2}}\sum_{m_{2}=|\ell_{2}-
\frac{1}{2}|}^{\ell_{2}+\frac{1}{2}}\left|a^{(\ell_{1},\ell_{2})}_{(m_{1},m_{2}),\ell}\right|^{2}
\Phi^{m_{1},m_{2}}_{\ell}
\end{equation}
where the coefficients $a^{(\ell_{1},\ell_{2})}_{(m_{1},m_{2}),\ell}$ are given by
\begin{equation}\label{eqn:recursioncoef}
a^{(\ell_{1},\ell_{2})}_{(m_{1},m_{2}),\ell}=\sum_{j_{1},j_{2},i_{1},i_{2},n_{1},n_{2}}
C^{\ell_{1},\ell_{2},\ell}_{j_{1},j_{2},\ell}C^{\frac{1}{2},\frac{1}{2},0}_{i_{1},i_{2},0}
C^{\ell_{1},\frac{1}{2},m_{1}}_{j_{1},i_{1},n_{1}}C^{\ell_{2},\frac{1}{2},m_{2}}_{j_{2},i_{2},n_{2}}
C^{m_{1},m_{2},\ell}_{n_{1},n_{2},\ell}.
\end{equation}
where the sum is taken  over
\begin{equation}\label{Restrictions}
|j_{1}|\le\ell_{1}, \quad |j_{2}| \le\ell_{2}, \quad |i_{1}|\le\frac{1}{2}, 
 \quad |i_{2}|\le\frac{1}{2}, \quad |n_{1}| \le m_{1}\quad\mbox{and}\quad|n_{2}|\le m_{2}.
\end{equation}
Moreover, $a^{(\ell_{1},\ell_{2})}_{(\ell_{1}+1/2,\ell_{2}+1/2),\ell}\ne0$. Note that the sum in (\ref{eqn:recursioncoef}) is a double sum because of 
Theorem \ref{thm:cgbasis}.
\end{prop}
Proposition \ref{prop:recursion} should be compared to Theorem 5.2 of \cite{PacharoniTirao2003}, where a similar calculation is given for the case $(\SU(3),\U(2))$.

\begin{proof}
On the one hand the representation $T^{\ell_{1},\ell_{2}}\otimes
T^{\frac{1}{2},\frac{1}{2}}$ can be written as a sum of at most 4 irreducible $U$-representations that contain the representation $T^{\ell}$ upon restriction to $K_{*}$. On the other hand we can find a `natural' copy $\scrH^{\ell}$ of $H^{\ell}$ in the space $H^{\ell_{1},\ell_{2}}\otimes H^{\frac{1}{2},\frac{1}{2}}$ that is invariant under the $K_{*}$-action. Projection onto this space transfers via $\alpha$, defined below, to a linear combination of projections on the spaces $H^{\ell}$ in the irreducible summands. The coefficients can be calculated in terms of Clebsch-Gordan coefficients and these in turn give rise to the recurrence relation. The details are as follows.

Consider the $U$-representation $T^{\ell_{1},\ell_{2}}\otimes
T^{\frac{1}{2},\frac{1}{2}}$ in the space $H^{\ell_{1},\ell_{2}}\otimes
H^{\frac{1}{2},\frac{1}{2}}$. By Theorem \ref{thm:cgbasis} we have
$$\alpha:H^{\ell_{1},\ell_{2}}\otimes
H^{\frac{1}{2},\frac{1}{2}}\to\bigoplus_{m_{1}=
|\ell_{1}-\frac{1}{2}|}^{\ell_{1}+\frac{1}{2}}\bigoplus_{m_{2}=|\ell_{2}
-\frac{1}{2}|}^{\ell_{2}+\frac{1}{2}}H^{m_{1},m_{2}}$$
which is a $U$-intertwiner given by
$$\alpha:\left(\psi^{\ell_{1}}_{j_{1}}\otimes\psi^{\ell_{2}}_{j_{2}}\right)
\otimes\left(\psi^{\frac{1}{2}}_{i_{1}}\otimes\psi^{\frac{1}{2}}_{i_{2}}\right)\mapsto
\sum_{m_{1}=|\ell_{1}-\frac{1}{2}|}^{\ell_{1}+\frac{1}{2}}
\sum_{n_{1}=-m_{1}}^{m_{1}}\sum_{m_{2}=|\ell_{2}-\frac{1}{2}|}^{\ell_{2}
+\frac{1}{2}}\sum_{n_{2}=-m_{2}}^{m_{2}}C^{\ell_{1},
\frac{1}{2},m_{1}}_{j_{1},i_{1},n_{1}}C^{\ell_{2},\frac{1}{2},m_{2}}_{j_{2},i_{2},n_{2}}
\psi^{m_{1}}_{n_{1}}\otimes\psi^{m_{2}}_{n_{2}}.
$$
Let $\scrH^{\ell}\subset H^{\ell_{1},\ell_{2}}\otimes
H^{\frac{1}{2},\frac{1}{2}}$ be the space that is spanned by the vectors
$$\left\{\phi^{\ell_{1},\ell_{2}}_{\ell,j}\otimes\phi^{\frac{1}{2},\frac{1}{2}}_{0,0}:-\ell\le
j\le \ell\right\}.$$
The element
$\phi^{\ell_{1},\ell_{2}}_{\ell,j}\otimes\phi^{\frac{1}{2},\frac{1}{2}}_{0,0}$
maps to
\begin{multline*}
\sum_{j_{1}=-\ell_{1}}^{\ell_{1}}\sum_{j_{2}=-\ell_{2}}^{\ell_{2}}\sum_{i_{1}=-\frac{1}{2}}^{\frac{1
}{2}}\sum_{i_{2}=-\frac{1}{2}}^{\frac{1}{2}}\sum_{m_{1}=|\ell_{1}-\frac{1}{2}|}^{\ell_{1}+\frac{1}{2
}}\sum_{n_{1}=-m_{1}}^{m_{1}}\sum_{m_{2}=|\ell_{2}-\frac{1}{2}|}^{\ell_{2}+\frac{1}{2}}\sum_{n_{2}
=-m_{2}}^{m_{2}}\sum_{p=|m_{1}-m_{2}|}^{m_{1}+m_{2}}\sum_{u=-p}^{p} \\
C^{\ell_{1},\ell_{2},\ell}_{j_{1},j_{2},j}C^{\frac{1}{2},\frac{1}{2},0}_{i_{1},i_{2},0}C^{\ell_{1},
\frac{1}{2},m_{1}}_{j_{1},i_{1},n_{1}}C^{\ell_{2},\frac{1}{2},m_{2}}_{j_{2},i_{2},n_{2}}
C^{m_{1},m_{2},p}_{n_{1},n_{2},u}\phi^{m_{1},m_{2}}_{p,u}.
\end{multline*}
Note that $u=n_{1}+n_{2}=j_{1}+i_{1}+j_{2}+i_{2}=j$, so the last sum can be omitted. Also, since $\alpha$ is a $K_{*}$-intertwiner, we must have $p=\ell$. For every pair $(m_{1},m_{2})$ we have a projection
$$P^ {(m_{1},m_{2})}_{\ell}:\bigoplus_{m_{1}=
|\ell_{1}-\frac{1}{2}|}^{\ell_{1}+\frac{1}{2}}\bigoplus_{m_{2}=|\ell_{2}
-\frac{1}{2}|}^{\ell_{2}+\frac{1}{2}}H^{m_{1},m_{2}}\to\bigoplus_{m_{1}=
|\ell_{1}-\frac{1}{2}|}^{\ell_{1}+\frac{1}{2}}\bigoplus_{m_{2}=|\ell_{2}
-\frac{1}{2}|}^{\ell_{2}+\frac{1}{2}}H^{m_{1},m_{2}}$$ onto the $\ell$-isotypical summand in the summand $H^{m_{1},m_{2}}$. Hence
\begin{multline*}P^{m_{1},m_{2}}_{\ell}(\alpha(\phi^{\ell_{1},\ell_{2}}_{\ell,j}\otimes\phi^{\frac{1}{2},\frac{1}{2}}_{0,0}))=\\
\sum_{j_{1}=-\ell_{1}}^{\ell_{1}}\sum_{j_{2}=-\ell_{2}}^{\ell_{2}}\sum_{i_{1}=-\frac{1}{2}}^{\frac{1
}{2}}\sum_{i_{2}=-\frac{1}{2}}^{\frac{1}{2}}\sum_{n_{1}=-m_{1}}^{m_{1}}\sum_{n_{2}
=-m_{2}}^{m_{2}}C^{\ell_{1},\ell_{2},\ell}_{j_{1},j_{2},j}C^{\frac{1}{2},\frac{1}{2},0}_{i_{1},i_{2},0}C^{\ell_{1},
\frac{1}{2},m_{1}}_{j_{1},i_{1},n_{1}}C^{\ell_{2},\frac{1}{2},m_{2}}_{j_{2},i_{2},n_{2}}
C^{m_{1},m_{2},\ell}_{n_{1},n_{2},j}\phi^{m_{1},m_{2}}_{\ell,j}.
\end{multline*}
The map $P^{m_{1},m_{2}}_{\ell}\circ\alpha$ is a $K_{*}$-intertwiner so Schur's lemma implies that
$$
\sum_{j_{1}=-\ell_{1}}^{\ell_{1}}\sum_{j_{2}=-\ell_{2}}^{\ell_{2}}\sum_{i_{1}=-\frac{1}{2}}^{\frac{1
}{2}}\sum_{i_{2}=-\frac{1}{2}}^{\frac{1}{2}}\sum_{n_{1}=-m_{1}}^{m_{1}}\sum_{n_{2}
=-m_{2}}^{m_{2}}C^{\ell_{1},\ell_{2},\ell}_{j_{1},j_{2},j}C^{\frac{1}{2},\frac{1}{2},0}_{i_{1},i_{2},0}C^{\ell_{1},
\frac{1}{2},m_{1}}_{j_{1},i_{1},n_{1}}C^{\ell_{2},\frac{1}{2},m_{2}}_{j_{2},i_{2},n_{2}}
C^{m_{1},m_{2},\ell}_{n_{1},n_{2},j}
$$
is independent of $j$. Hence it is equal to $a^{(\ell_{1},\ell_{2})}_{(m_{1},m_{2}),\ell}$, taking $j=\ell$. We have $$\alpha(\phi^{\ell_{1},\ell_{2}}_{\ell,j}\otimes\phi^{\frac{1}{2},\frac{1}{2}}_{0,0})=\sum_{m_{1},m_{2}}a^{(\ell_{1},\ell_{2})}_{(m_{1},m_{2}),j}\phi^{m_{1},m_{2}}_{\ell,j}.$$ 
Moreover, the map
$$P= \hskip-5pt \sum_{m_{1}=|\ell_{1}-1/2|}^{\ell_{1}+1/2}\sum_{m_{2}=|\ell_{2}-1/2|}^{\ell_{2}+1/2}P^{m_{1},m_{2}}_{\ell}\circ\alpha:H^{\ell_{1},\ell_{2}}\otimes H^{\frac{1}{2},\frac{1}{2}}\to\bigoplus_{m_{1}=
|\ell_{1}-\frac{1}{2}|}^{\ell_{1}+\frac{1}{2}}\bigoplus_{m_{2}=|\ell_{2}
-\frac{1}{2}|}^{\ell_{2}+\frac{1}{2}}H^{m_{1},m_{2}}$$
is a $K_{*}$-intertwiner. To show that it is not the trivial map we note that $a^{(\ell_{1},\ell_{2})}_{(\ell_{1}+\frac12,\ell_{2}+\frac12),\ell}$ is non-zero. Indeed, the equalities
$$C^{\ell_{1},\ell_{2},\ell}_{j_{1},j_{2},\ell}=\frac{(-1)^{\ell_{1}-j_{1}}(\ell+\ell_{2}-\ell_{1})!}{(\ell_{1}+\ell_{2}+\ell+1)!\Delta(\ell_{1},\ell_{2},\ell)}\left[\frac{(2\ell+1)(\ell_{1}+j_{1})!(\ell_{2}+\ell-j_{1})!}{(\ell_{1}-j_{1})!(\ell_{2}-\ell+j_{1})!}\right]^{1/2},$$
$$C^{\ell_{1},\frac{1}{2},\ell_{1}+\frac{1}{2}}_{j_{1},\frac{1}{2},j_{1}+\frac{1}{2}}=\left[\frac{\ell_{1}+j_{1}+1}{2\ell_{1}+1}\right]^{1/2}\quad\mbox{and}\quad C^{\ell_{1},\frac{1}{2},\ell_{1}+\frac{1}{2}}_{j_{1},-\frac{1}{2},j_{1}-\frac{1}{2}}=\left[\frac{\ell_{1}-j_{1}+1}{2\ell_{1}+1}\right]^{1/2},$$
where $\Delta(\ell_{1},\ell_{2},\ell)$ is a positive function, can be found in \cite{Vilenkin}*{Ch. 8} and plugging these into the formula for $a^{(\ell_{1},\ell_{2})}_{(\ell_{1}+\frac12,\ell_{2}+\frac12),\ell}$ shows that it is the sum of positive numbers, hence it is non-zero.

We conclude that $P$ is non-trivial, so its restriction to $\scrH^{\ell}$ is an isomorphism and it intertwines the $K_{*}$-action. It maps $K_{*}$-isotypical summands to $K_{*}$-isotypical summands. Hence $\alpha|\scrH^{\ell}=P|\scrH^{\ell}$. Define
$$\gamma^{\ell_{1},\ell_{2}}_{\ell}:H^{\ell}\to H^{\ell_{1},\ell_{2}}\otimes H^{\frac{1}{2},\frac{1}{2}}:\psi^{\ell}_{j}
\mapsto\phi^{\ell_{1},\ell_{2}}_{\ell,j}\otimes\phi^{\frac{1}{2},\frac{1}{2}}_{0,0}.$$
This is a $K$-intertwiner. It follows that
\begin{equation}\label{intertwiner:gamma} \alpha\circ\gamma^{\ell_{1},\ell_{2}}_{\ell}=\sum_{m_{1},m_{2}}a^{(\ell_{1},\ell_{2})}_{(m_{1},m_{2}),\ell}\beta^{m_{1},m_{2}}_{\ell}.
\end{equation}
Define the $\End(H^{\ell})$-valued function
$$\Psi_{\ell_{1},\ell_{2}}^{\ell}:U\to\End(H^{\ell}):x\mapsto(\gamma^{\ell_{1},\ell_{2}}_{\ell})^{*}
\circ(T^{\ell_{1},\ell_{1}}\otimes
T^{\frac{1}{2},\frac{1}{2}})(x)\circ\gamma^{\ell_{1},\ell_{2}}_{\ell}.$$
Note that $\Psi_{\ell_{1},\ell_{2}}^{\ell}(x)=\varphi(x)\Phi_{\ell_{1},\ell_{2}}^{\ell}(x)$. On the other hand we have $T^{\ell_{1},\ell_{2}}\otimes T^{\frac{1}{2},\frac{1}{2}}=\alpha\circ\left(\bigoplus_{m_{1},m_{2}}T^{m_{1},m_{2}}\right)\circ\alpha$. Together with (\ref{intertwiner:gamma}) this yields
$$\Psi^{\ell}_{\ell_{1},\ell_{2}}=\sum_{m_{1},m_{2}}|a^{(\ell_{1},\ell_{2})}_{(m_{1},m_{2}),\ell}|^{2}(\beta^{\ell_{1},\ell_{2}}_{\ell})^{*}\circ T^{m_{1},m_{2}}\circ\beta^{m_{1},m_{2}}_{\ell}.$$
Hence the result.
\end{proof}

In Figure \ref{fig:tensordecomp} we have depicted the representations
$(\frac{1}{2},\frac{1}{2})$ and $(\frac{5}{2},3)$ with black nodes. The
tensor product decomposes into the four types
$(\frac{5}{2}\pm\frac{1}{2},3\pm\frac{1}{2})$ which are indicated with the
white nodes.
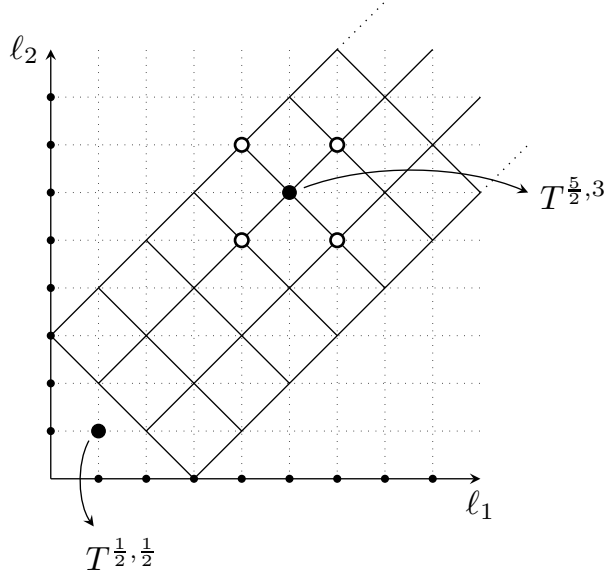
\begin{figure}
\begin{center}
\resizebox{.5\textwidth}{!}{
\begin{tikzpicture}[scale=0.5,>=stealth]
\draw[->] (0,0) -- (9,0) node[below] {$\ell_1$};
\draw[->] (0,0) -- (0,9) node[left] {$\ell_2$};

\draw[very thin,dotted] (0,1) -- (9,1);
\draw[very thin,dotted] (0,2) -- (9,2);
\draw[very thin,dotted] (0,3) -- (9,3);
\draw[very thin,dotted] (0,4) -- (9,4);
\draw[very thin,dotted] (0,5) -- (9,5);
\draw[very thin,dotted] (0,6) -- (9,6);
\draw[very thin,dotted] (0,7) -- (9,7);
\draw[very thin,dotted] (0,8) -- (9,8);

\draw[very thin,dotted] (1,0) -- (1,9);
\draw[very thin,dotted] (2,0) -- (2,9);
\draw[very thin,dotted] (3,0) -- (3,9);
\draw[very thin,dotted] (4,0) -- (4,9);
\draw[very thin,dotted] (5,0) -- (5,9);
\draw[very thin,dotted] (6,0) -- (6,9);
\draw[very thin,dotted] (7,0) -- (7,9);
\draw[very thin,dotted] (8,0) -- (8,9);

\draw[fill=black] (0,1) circle (2pt);
\draw[fill=black] (0,2) circle (2pt);
\draw[fill=black] (0,3) circle (2pt);
\draw[fill=black] (0,4) circle (2pt);
\draw[fill=black] (0,5) circle (2pt);
\draw[fill=black] (0,6) circle (2pt);
\draw[fill=black] (0,7) circle (2pt);
\draw[fill=black] (0,8) circle (2pt);

\draw[fill=black] (1,0) circle (2pt);
\draw[fill=black] (2,0) circle (2pt);
\draw[fill=black] (3,0) circle (2pt);
\draw[fill=black] (4,0) circle (2pt);
\draw[fill=black] (5,0) circle (2pt);
\draw[fill=black] (6,0) circle (2pt);
\draw[fill=black] (7,0) circle (2pt);
\draw[fill=black] (8,0) circle (2pt);

\draw[] (0,3) -- (3,0);

\draw[] (0,3) -- (6,9);
\draw[] (3,0) -- (9,6);

\draw[] (1,2) -- (8,9);
\draw[] (2,1) -- (9,8);

\draw[dotted] (6,9) -- (7,10);
\draw[dotted] (9,6) -- (10,7);
\draw[] (6,9) -- (9,6);

\draw[] (1,4) -- (4,1);
\draw[] (2,5) -- (5,2);
\draw[] (3,6) -- (6,3);
\draw[] (4,7) -- (7,4);
\draw[] (5,8) -- (8,5);


\draw[fill=black] (1,1) circle (4pt);
\draw[fill=black] (5,6) circle (4pt);


\draw[fill=white, thick] (6,7) circle (4pt);
\draw[fill=white, thick] (6,5) circle (4pt);
\draw[fill=white, thick] (4,7) circle (4pt);
\draw[fill=white, thick] (4,5) circle (4pt);


\path[thin] (5.3,6.1) edge[->,bend left=22,looseness=0.8] (10,6);
\draw (10,6) node[right] {$T^{\frac52,3}$};

\path[thin] (0.8,0.8) edge[->,bend right=32,looseness=0.8] (0.9,-1);
\draw (1.5,-1) node[below] {$T^{\frac12,\frac12}$};


\end{tikzpicture}}
\end{center}
\caption{Plot of how the tensor product $T^{\frac{5}{2},3}\otimes T^{\frac{1}{2},\frac{1}{2}}$
splits into irreducible summands.}\label{fig:tensordecomp}
\end{figure}

\begin{cor} Given a spherical function $\Phi_{\ell_{1},\ell_{2}}^{\ell}$
there exist $2\ell+1$ elements $q_{\ell_{1},\ell_{2}}^{\ell,j}$,
$j\in\{-\ell,-\ell+1,\ldots,\ell\}$ in $\bbC[\varphi]$ such that
\begin{equation}\label{eqn:poly}
\Phi_{\ell_{1},\ell_{2}}^{\ell}=\sum_{j=-\ell}^{\ell}q_{\ell_{1},\ell_{2}}^{\ell,j}
\Phi_{(\ell+j)/2,(\ell-j)/2}^{\ell}.
\end{equation}
The degree of $q_{\ell_{1},\ell_{2}}^{\ell,j}$ is $\ell_{1}+\ell_{2}-\ell$.
\end{cor}
\begin{proof}
We prove this by induction on $\ell_{1}+\ell_{2}$. If $\ell_{1}+\ell_{2}=\ell$
then the statement is true with the polynomials $q_{(\ell-k)/2,(\ell+k)/2}^{\ell,j}
=\delta_{j,k}$. Suppose $\ell_{1}+\ell_{2}>\ell$
and that the statement holds for $(\ell_{1}',\ell_{2}')$
with $\ell\le\ell_{1}'+\ell_{2}'<\ell_{1}+\ell_{2}$. We can write $|a^{(\ell_{1}-1/2,\ell_{2}-1/2)}_{(\ell_{1},\ell_{2}),\ell}|^{2}\Phi_{\ell_{1},\ell_{2}}$ as
$$\varphi\cdot\Phi_{\ell_{1}-\frac{1}{2},\ell_{2}
-\frac{1}{2}}-|a_{(\ell_{1}-1,\ell_{2}),\ell}^{\ell_{1}-
\frac{1}{2},\ell_{2}-\frac{1}{2}}|^{2}\Phi_{\ell_{1}-1,\ell_{2}}-|a_{(\ell_{1},\ell_{2}-1),\ell}^{\ell_{1}-\frac{1}{2},
\ell_{2}-\frac{1}{2}}|^{2}\Phi_{\ell_{1},\ell_{2}-1}
-|a_{(\ell_{1}-1,\ell_{2}-1),\ell}^{\ell_{1}-\frac{1}{2},
\ell_{2}-\frac{1}{2}}|^{2}\Phi_{\ell_{1}-1,\ell_{2}-1}
$$
by means of Proposition \ref{prop:recursion}. The result follows from the induction hypothesis and $a^{(\ell_{1}-1/2,\ell_{2}-1/2)}_{(\ell_{1},\ell_{2}),\ell}\ne0$.
\end{proof}

\begin{remark}
The fact that these functions $q_{\ell_{1},\ell_{2}}^{\ell,j}$ are polynomials in $\cos(t)$ has also been shown by Koornwinder in Theorem 3.4 of \cite{Koornwinder85} using different methods.
\end{remark}


\section{Restricted Spherical Functions}\label{sec:restricted_sf}\label{4}

For the restricted spherical functions $\Phi_{\ell_{1},\ell_{2}}^{\ell}:A_{*}\to\End(H^{\ell})$ we define a
pairing.
\begin{equation}\label{eqn:pairing2}
\begin{split}
\langle\Phi,\Psi\rangle_{A_{*}}&=\frac{2}{\pi}\tr\left(\int_{A_{*}}\Phi(a)\left(\Psi(a)\right)^{*}|D_{*}(a)|da\right)\\
\end{split}
\end{equation}
where $D_{*}(a_{t})=\sin^{2}(t)$, see \cite{HelgasonDGSS}. In \cite{Koornwinder85}*{Prop. 2.2} it is shown that on $A_{*}$ the following orthogonality relations hold for the restricted spherical functions.
\begin{prop}\label{prop:orthosf}
The spherical functions on $U$ of type $\ell$, when restricted to $A_{*}$,
are orthogonal with respect to (\ref{eqn:pairing2}). In fact, we have
\begin{equation}\langle\Phi_{\ell_{1},\ell_{2}}^{\ell},\Phi_{\ell_{1}',\ell_{2}'}^{\ell}\rangle_{A_{*}}=\frac{(2\ell+1)^{2}}{(2\ell_{1}+1)(2\ell_{2}+1)}\delta_{\ell_{1},\ell_{1}'}\delta_{\ell_{2},\ell_{2}'}
\end{equation}
\end{prop}
This is a direct consequence of the Schur orthogonality relations and the integral formula corresponding to the $U=K_{*}A_{*}K_{*}$ decomposition.

The parametrization of the $U$-types that contain a fixed $K$-type $\ell$
is given by (\ref{eqn:parametrization}). For later purposes we reparamatrize (\ref{eqn:parametrization}) by the function
$\zeta:\bbN\times\{-\ell,\ldots,\ell\}\to\frac{1}{2}\bbN\times\frac{1}{2}\bbN$
given by
$$\zeta(d,k)=\left(\frac{d+\ell+k}{2},\frac{d+\ell-k}{2}\right).$$
This new parametrization is pictured in Figure \ref{figure1}. For each degree $d$ have $2\ell+1$ spherical functions. By Proposition \ref{prop:diag} the
restricted spherical functions take their values in the vector space
$\End_{M}(H^{\ell})$ which is $2\ell+1$-dimensional. The appearance of the spherical
functions in $2\ell+1$-tuples gives rise to the following definition.

\begin{figure}
\begin{center}
\resizebox{.424\textwidth}{!}{
\begin{tikzpicture}[scale=0.5,>=stealth]
\draw[->] (0,0) -- (9,0) node[below] {$\ell_1$};
\draw[->] (0,0) -- (0,9) node[left] {$\ell_2$};

\draw[very thin,dotted] (0,1) -- (9,1);
\draw[very thin,dotted] (0,2) -- (9,2);
\draw[very thin,dotted] (0,3) -- (9,3);
\draw[very thin,dotted] (0,4) -- (9,4);
\draw[very thin,dotted] (0,5) -- (9,5);
\draw[very thin,dotted] (0,6) -- (9,6);
\draw[very thin,dotted] (0,7) -- (9,7);
\draw[very thin,dotted] (0,8) -- (9,8);

\draw[very thin,dotted] (1,0) -- (1,9);
\draw[very thin,dotted] (2,0) -- (2,9);
\draw[very thin,dotted] (3,0) -- (3,9);
\draw[very thin,dotted] (4,0) -- (4,9);
\draw[very thin,dotted] (5,0) -- (5,9);
\draw[very thin,dotted] (6,0) -- (6,9);
\draw[very thin,dotted] (7,0) -- (7,9);
\draw[very thin,dotted] (8,0) -- (8,9);

\draw[fill=black] (0,1) circle (2pt);
\draw[fill=black] (0,2) circle (2pt);
\draw[fill=black] (0,3) circle (2pt);
\draw[fill=black] (0,4) circle (2pt);
\draw[fill=black] (0,5) circle (2pt);
\draw[fill=black] (0,6) circle (2pt);
\draw[fill=black] (0,7) circle (2pt);
\draw[fill=black] (0,8) circle (2pt);

\draw[fill=black] (1,0) circle (2pt);
\draw[fill=black] (2,0) circle (2pt);
\draw[fill=black] (3,0) circle (2pt);
\draw[fill=black] (4,0) circle (2pt);
\draw[fill=black] (5,0) circle (2pt);
\draw[fill=black] (6,0) circle (2pt);
\draw[fill=black] (7,0) circle (2pt);
\draw[fill=black] (8,0) circle (2pt);

\draw[thick] (0,3) -- (3,0);

\draw[thick] (0,3) -- (6,9);
\draw[thick] (3,0) -- (9,6);

\draw[dotted] (6,9) -- (7,10);
\draw[dotted] (9,6) -- (10,7);

\draw[] (9,6) -- (6,9);

\draw[] (1,2) -- (8,9);
\draw[] (2,1) -- (9,8);

\draw[] (1,4) -- (4,1);
\draw[] (2,5) -- (5,2);
\draw[] (3,6) -- (6,3);
\draw[] (4,7) -- (7,4);
\draw[] (5,8) -- (8,5);

\draw[->] (1/4+1/4,7/4+1/4) -- (7/4+1/4,1/4+1/4) node[below,midway] {$k$};
\draw[->] (2-1/4,5+1/4) -- (4-1/4,7+1/4) node[above,midway] {$d$};

\draw[fill=black] (6,3) circle (2pt);
\draw (5.9,2.6)  node[right] {$\Phi_{3,\frac32}^{\frac32}$};

\end{tikzpicture}}
\end{center}
\caption{Another parametrization of the pairs $(\ell_{1},\ell_{2})$ containing $\ell$; the steps of the ladder are paramatrized by $d$, the position on a given step by $k$.}\label{figure1}
\end{figure}
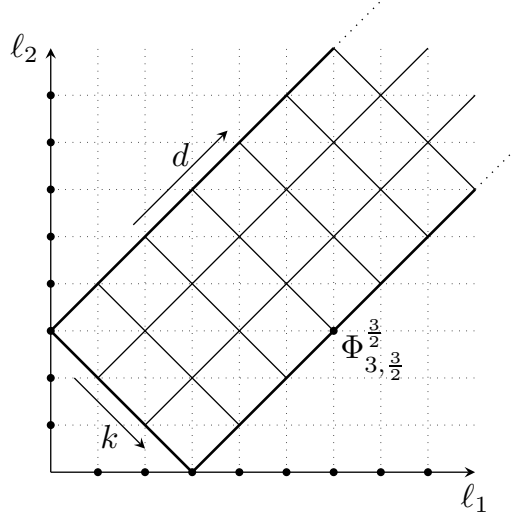

\begin{definition}\label{def:fullsf}
Fix a $K$-type $\ell\in\frac{1}{2}\bbN$ and a degree $d\in\bbN$. The
function $\Phi_{d}^{\ell}:A_{*}\to\End(H^{\ell})$ is defined by
associating to each point $a\in A_{*}$ a matrix $\Phi_{d}^{\ell}(a)$
whose $j$-th row is the vector $\Phi_{\zeta(d,j)}^{\ell}(a)$. More precise, we have
\begin{equation}\label{eq:fullsf coef}
\left(\Phi_{d}^{\ell}(a)\right)_{p,q}=\left(\Phi_{\zeta(d,p)}^{\ell}(a)\right)_{q,q}\quad\mbox{for all $a\in A_{*}$}.
\end{equation}
The function $\Phi_{d}^{\ell}$ is called the full spherical function of type $\ell$ and degree $d$.
\end{definition}

Let $A_{*}'$ be the open subset $\{a_{t}\in A_{*}: t\not\in\pi\bbZ\}$.
This is the regular part of $A_{*}$. The following Proposition is shown in \cite{Koornwinder85}*{Prop. 3.2}. In Proposition \ref{prop:invertible2} we prove this result independently for general points in $A_{*}$ in a different way.
\begin{prop}\label{prop:invertible}
The full spherical function $\Phi_{0}^{\ell}$ of type $\ell$ and degree
0 has the property that its  restriction to $A_{*}'$ is invertible.
\end{prop}

\begin{definition}\label{def:polQ}
Fix a $K$-type $\ell\in\frac{1}{2}\bbN$ and a degree $d\in\bbN$.
Define the function
\begin{equation}
Q_{d}^{\ell}:A_{*}'\to\End(H^{\ell}):a\mapsto \Phi_{d}^{\ell}(a)(\Phi_{0}^{\ell}(a))^{-1}.
\end{equation}
The  $j$-th row is denoted by $Q_{\zeta(d,j)}^{\ell}(a)$. $Q_{d}^{\ell}$ is called the full spherical
polynomial of type $\ell$ and degree $d$.
\end{definition}

The functions $Q_{d}^{\ell}$ and $Q_{\zeta(d,k)}^{\ell}$ are
polynomials because $\left(Q_{d}^{\ell}\right)_{p,q}=q^{\ell,q}_{\zeta(d,p)}(\varphi)$. The degree of each row of $Q_{d}^{\ell}$ is $d$ which justifies the name we have given these functions in Definition \ref{def:polQ}.

We shall show that the functions $\Phi_{d}^{\ell}$ and $Q_{d}^{\ell}$ satisfy orthogonality
relations that come from (\ref{eqn:pairing2}). We start with the
$\Phi_{d}^{\ell}$. This function encodes $2\ell+1$ restricted spherical
functions and to capture the orthogonality relations of
(\ref{eqn:pairing2}) we need a matrix valued inner product.

\begin{definition}
Let $\Phi,\Psi$ be $\End(H^{\ell})$-valued
functions on $A_{*}$. Define
\begin{equation}\label{def:pairing2}
\langle\Phi,\Psi\rangle:=\frac{2}{\pi}\int_{A_{*}}\Phi(a)\left(\Psi(a)\right)^{*}|D_{*}(a)|da.
\end{equation}
\end{definition}

\begin{prop}\label{prop:pairing2}
The pairing defined by (\ref{def:pairing2}) is a matrix-valued inner
product. The functions $\Phi_{d}^{\ell}$ with $d\in\bbN$ form an
orthogonal family with respect to this inner product.
\end{prop}
\begin{proof}The pairing satisfies all the linearity conditions of a matrix valued inner product.
Moreover we have $\Phi(a)\left(\Phi(a)\right)^{*}|D_{*}(a)|\ge0$ for all $a\in A_{*}$.
If $\langle\Phi,\Phi\rangle=0$ then $\Phi\Phi^{*}=0$ from which it follows that $\Phi=0$.
Hence the pairing is an inner product. The orthogonality follows from the formula
$$\left(\langle\Phi_{d},\Psi_{d'}\rangle\right)_{p,q}=\langle\Phi_{\zeta(d,p)}^{\ell},
\Phi_{\zeta(d',q)}^{\ell}\rangle_{A_{*}}=\delta_{d,d'}\delta_{p,q}\frac{(2\ell+1)^{2}}{(d+\ell+p+1)(d+\ell-p+1)}$$
and Proposition \ref{prop:orthosf}.
\end{proof}

Define
\begin{equation}\label{eqn:weightonA}
V^{\ell}(a)=\Phi_{0}^{\ell}(a)\left(\Phi_{0}^{\ell}(a)\right)^{*}|D_{*}(a)|,
\end{equation}
with $D_{*}(a_{t})=\sin^{2}t$. This is a weight matrix and we have the following corollary.
\begin{cor}
Let $Q$ and $R$ be $\End(H^{\ell})$-valued functions on $A_{*}$ and define
the matrix valued paring with respect to the weight $V^{\ell}$ by
\begin{equation}
\langle Q,R\rangle_{V^{\ell}}=\int_{A_{*}}Q(a)V^{\ell}(a)\left(R(a)\right)^{*}da.
\end{equation}
This pairing is a matrix valued inner product and the functions
$Q_{d}^{\ell}$ form an orthogonal family for this inner product.
\end{cor}

The functions $\Phi_{d}^{\ell}$ and $Q_{d}^{\ell}$ being defined, we
can now transfer the recurrence relations of Proposition
\ref{prop:recursion} to these functions. Let $E_{i,j}$ be the elementary matrix with zeros everywhere except for
the $(i,j)$-th spot, where it has a one. If we write $E_{i,j}$ with
$|i|>\ell$ or $|j|>\ell$ then we mean the zero matrix.

\begin{theorem}\label{thm:three_term_for_Qd}
Fix $\ell\in\frac{1}{2}\bbN$ and define the matrices $A_{d},B_{d}$ and
$C_{d}$ by
\begin{equation}\label{eq:coefficients_three_term}
\begin{split}
A_{d}&=\sum_{k=-\ell}^{\ell}|a^{\zeta(d,k)}_{\zeta(d+1,k),\ell}|^{2}E_{k,k},\\
B_{d}&=\sum_{k=-\ell}^{\ell}\left(\left|a^{\zeta(d,k),}_{\zeta(d,k+1),\ell}
\right|^{2}E_{k,k+1}+\left|a^{\zeta(d,k)}_{\zeta(d,k-1),\ell}\right|^{2}E_{k,k-1}\right),\\
C_{d}&=\sum_{k=-\ell}^{\ell}|a^{\zeta(d,k)}_{\zeta(d-1,k),\ell}|^{2}E_{k,k}.
\end{split}
\end{equation}
For $a\in
A_{*}$ we have
\begin{equation}\label{recrelfs}
\varphi(a)\cdot\Phi_{d}^{\ell}(a)=A_{d}\Phi_{d+1}^{\ell}(a)+B_{d}\Phi_{d}^{\ell}(a)+C_{d}\Phi_{d-1}^{\ell}(a)
\end{equation}
and similarly
\begin{equation}\label{recrelpol}
\varphi(a)\cdot Q_{d}^{\ell}(a)=A_{d}Q_{d+1}^{\ell}(a)+B_{d}Q_{d}^{\ell}(a)+C_{d}Q_{d-1}^{\ell}(a).
\end{equation}
Note $A_d \in GL_{2\ell+1}(\bbR)$.
\end{theorem}
\begin{proof}
It is clear that (\ref{recrelpol}) follows from (\ref{recrelfs}) by multiplying on the
right with the inverse of $\Phi_{0}^{\ell}$. To prove (\ref{recrelfs}) we look at the rows.
Let $p\in\{-\ell,-\ell+1,\ldots,\ell\}$ and multiply (\ref{recrelfs}) on the left by $E_{p,p}$ to pick out the $p$-th row. The left hand side gives $\varphi(a)E_{p,p}\Phi_{d}^{\ell}(a)$ while the right hand side gives
$$|a^{\zeta(d,p)}_{\zeta(d+1,p),\ell}|^{2}E_{p,p}\Phi_{d+1}^{\ell}(a)+|a^{\zeta(d,p)}_{\zeta(d,p+1),\ell}|^{2}E_{p,p+1}\Phi_{d}^{\ell}(a)+
|a^{\zeta(d,p)}_{\zeta(d,p-1),\ell}|^{2}E_{p,p-1}\Phi_{d}^{\ell}(a)
+|a^{\zeta(d,p)}_{\zeta(d-1,p),\ell}|^{2}E_{p,p}\Phi_{d-1}^{\ell}(a).
$$
Now observe that these are equal by Proposition \ref{prop:recursion} and (\ref{eq:fullsf coef}). This proves the result since $p$ is arbitrary.
\end{proof}

Finally we discuss some symmetries of the full spherical functions. The Cartan involution corresponding to the pair $(U,K_{*})$ is the map
$\theta(k_{1},k_{2})=(k_{2},k_{1})$. The representation
$T^{\ell_{1},\ell_{2}}$ and $T^{\ell_{2},\ell_{1}}\circ\theta$ are
equivalent via the map
$\psi^{\ell_{1}}_{j_{1}}\otimes\psi^{\ell_{2}}_{j_{2}}\mapsto\psi^{\ell_{2}}_{j_{2}}
\otimes\psi^{\ell_{1}}_{j_{1}}$.
It follows that
$\theta^{*}\Phi_{\ell_{1},\ell_{2}}^{\ell}=\Phi_{\ell_{2},\ell_{1}}^{\ell}$.
This has the following effect on the full spherical functions
$\Phi_{d}^{\ell}$ from Definition \ref{def:fullsf}:
\begin{equation}\label{eqn:cartanfsf}
\theta^{*}\Phi_{d}^{\ell}=J\Phi_{d}^{\ell},
\end{equation}
where $J\in\End(H^{\ell})$ is given by $\psi^{\ell}_{j}\mapsto\psi^{\ell}_{-j}$.
The Weyl group $\mathcal{W}(U,K_{*})=\{1,s\}$ consists of the identity and the
reflection $s$ in
$0\in\laa_{*}$. The group $\mathcal{W}(U,K_{*})$ acts on $A_{*}$ and on the
functions on $A_{*}$ by pull-back.

\begin{lemma}
We have
$s^{*}\Phi_{\ell_{1},\ell_{2}}^{\ell}(a)=J\Phi_{\ell_{1},\ell_{2}}^{\ell}(a)$ for all $a\in A_{*}$.
The effect on the full spherical functions of type $\ell$ is
\begin{equation}\label{eqn:weylfsf}
s^{*}\Phi_{d}^{\ell}=\Phi_{d}^{\ell}J.
\end{equation}
\end{lemma}

\begin{proof}
This follows from (\ref{eqn:ressfGENERAL}) and the fact that
$C^{\ell_{1},\ell_{2},\ell}_{j_{1},j_{2},j}=(-1)^{\ell_{1}+\ell_{2}-\ell}C^{\ell_{1},\ell_{2},\ell}_{-j_{1},-j_{2},-j}$.
\end{proof}

\begin{prop}\label{prop:VandJcomm}
The functions $\Phi_{d}^{\ell}$ commute with $J$.
\end{prop}
\begin{proof}
The action of $\theta$ and $s_{\alpha}$ on $A_{*}$ is just taking the
inverse. Formulas (\ref{eqn:cartanfsf}) and (\ref{eqn:weylfsf}) now yield
the result.
\end{proof}


\section{The Weight Matrix}\label{sec:symmetries_of_W}\label{5}

We study the weight function $V^{\ell}:A_{*}\to\End(H^{\ell})$ defined in (\ref{eqn:weightonA}), in particular its symmetries and explicit expressions for its matrix elements.
First note that $V^{\ell}$ is real valued.
Indeed, $V^{\ell}$ commutes with $J$,

\begin{equation}\label{eq:V and J comm}
JV^{\ell}(a)J=J\Phi^{\ell}_{0}(a)J(J\Phi^{\ell}_{0}(a)J)^{*}|D_{*}(a)|=\Phi^{\ell}_{0}(a)\Phi_{0}^{\ell}(a)^{*}|D_{*}(a)|=V^{\ell}(a),
\end{equation}
since $J^{*}=J$ and $J^{2}=1$. This also shows that $V$ is real valued,
\begin{equation}\label{eq:Vreal}
\overline{V^{\ell}(a_{t})}=V^{\ell}(a_{-t})=JV^{\ell}(a_{t})J=V^{\ell}(a_{t}).
\end{equation}

\begin{lemma}\label{lemma:symmetries}
The weight has the symmetries $V^{\ell}_{p,q}=V^{\ell}_{q,p}=V^{\ell}_{-p,-q}=V^{\ell}_{-q,-p}$ for $p,q\in\{-\ell,\ldots,\ell\}$.
\end{lemma}
\begin{proof}
The first equallity follows since $V^{\ell}(a)$ is self adjoint by (\ref{eqn:weightonA}) and real valued by (\ref{eq:Vreal}). Since $V^{\ell}(a)$ commutes with $J$ we see that $V^{\ell}_{p,q}=V^{\ell}_{-p,-q}$.
\end{proof}
Set
\begin{equation}
v^{\ell}(a_{t})=\Phi^{\ell}_{0}(a)\Phi^{\ell}_{0}(a)^{*}
\end{equation}
so that $V^{\ell}(a_{t})=v^{\ell}(a_{t})|D_{*}(a_{t})|=v^{\ell}(a_{t})\sin^{2}t$. Note that for $-\ell\le p,q\le\ell$ the matrix coefficient
\begin{equation}
v^{\ell}(a_{t})_{p,q}=\tr\left(\Phi^{\ell}_{\frac{\ell+p}{2},\frac{\ell-p}{2}}(a_{t})\left(\Phi^{\ell}_{\frac{\ell+q}{2},\frac{\ell-q}{2}}(a_{t})\right)^{*}\right)
\end{equation}
is a linear combination of zonal spherical functions by the following lemma.

\begin{lemma}\label{lem: erik}
The function $U\to\bbC: x\mapsto\tr\left(\Phi^{\ell}_{\ell_{1},\ell_{2}}(x)\left(\Phi^{\ell}_{m_{1},m_{2}}(x)\right)^{*}\right)$ is a bi-$K$-invariant function and
\begin{equation}\label{eq:erik}
\tr\left(\Phi^{\ell}_{\ell_{1},\ell_{2}}(a_{t})\left(\Phi^{\ell}_{m_{1},m_{2}}(a_{t})\right)^{*}\right)=\sum_{n=\max(|\ell_{1}-m_{1}|,|\ell_{2}-m_{2}|)}^{\min(\ell_{1}+m_{1},\ell_{2}+m_{2})}c_{n}U_{2n}(\cos t)
\end{equation}
if $\ell_{1}+m_{1}-(\ell_{2}+m_{2})\in\bbZ$ and $\tr\left(\Phi^{\ell}_{\ell_{1},\ell_{2}}(a_{t})\left(\Phi^{\ell}_{m_{1},m_{2}}(a_{t})\right)^{*}\right)=0$ otherwise.
\end{lemma}
\begin{proof}
It follows from Property (2) that the function is bi-$K$-invariant, so it is natural to expand the function in terms of the zonal spherical functions $U_{2n}$ corresponding to the spherical representations $T^{n,n}$, $n\in\frac{1}{2}\bbN$. Since $T^{\ell_{1},\ell_{2}}$ is equivalent to its contragradient representation, we see that the only spherical functions occurring in the expansion of $\tr\left(\Phi^{\ell}_{\ell_{1},\ell_{2}}(x)\left(\Phi^{\ell}_{m_{1},m_{2}}(x)\right)^{*}\right)$ are the ones for which $(n,n)\in A=\{(n_{1},n_{2})\in\frac{1}{2}\bbN\times\frac{1}{2}\bbN:\ell_{i}+m_{i}-n_{i}\in\bbZ,|\ell_{1}-m_{1}|\le n_{1}\le\ell_{1}+m_{1},|\ell_{2}-m_{2}|\le n_{2}\le\ell_{2}+m_{2}\}$ since the right hand side corresponds to the tensorproduct decomposition $T^{\ell_{1},\ell_{2}}\otimes T^{m_{1},m_{2}}=\bigoplus_{(n_{1},n_{2})\in A}T^{n_{1},n_{2}}$, see Proposition \ref{prop:recursion} and Figure \ref{figure:decomp}.
\end{proof}

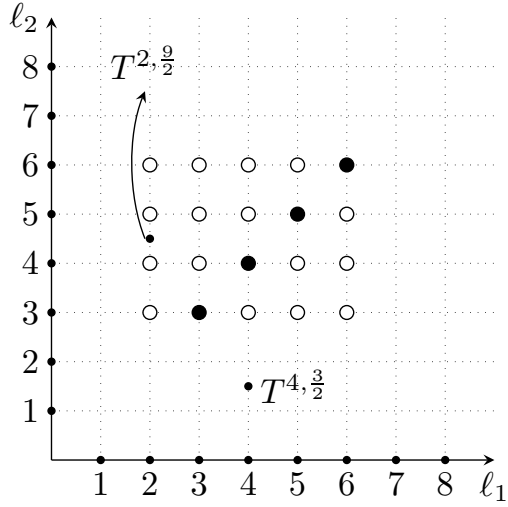
\begin{figure}
\begin{center}
\resizebox{.424\textwidth}{!}{
\begin{tikzpicture}[scale=0.5,>=stealth]
\draw[->] (0,0) -- (9,0) node[below] {$\ell_1$};
\draw[->] (0,0) -- (0,9) node[left] {$\ell_2$};

\draw[very thin,dotted] (0,1) -- (9,1);
\draw[very thin,dotted] (0,2) -- (9,2);
\draw[very thin,dotted] (0,3) -- (9,3);
\draw[very thin,dotted] (0,4) -- (9,4);
\draw[very thin,dotted] (0,5) -- (9,5);
\draw[very thin,dotted] (0,6) -- (9,6);
\draw[very thin,dotted] (0,7) -- (9,7);
\draw[very thin,dotted] (0,8) -- (9,8);

\draw[very thin,dotted] (1,0) -- (1,9);
\draw[very thin,dotted] (2,0) -- (2,9);
\draw[very thin,dotted] (3,0) -- (3,9);
\draw[very thin,dotted] (4,0) -- (4,9);
\draw[very thin,dotted] (5,0) -- (5,9);
\draw[very thin,dotted] (6,0) -- (6,9);
\draw[very thin,dotted] (7,0) -- (7,9);
\draw[very thin,dotted] (8,0) -- (8,9);

\draw[fill=black] (0,1) circle (2pt) node[left]{$1$};;
\draw[fill=black] (0,2) circle (2pt) node[left]{$2$};;
\draw[fill=black] (0,3) circle (2pt) node[left]{$3$};;
\draw[fill=black] (0,4) circle (2pt) node[left]{$4$};;
\draw[fill=black] (0,5) circle (2pt) node[left]{$5$};;
\draw[fill=black] (0,6) circle (2pt) node[left]{$6$};;
\draw[fill=black] (0,7) circle (2pt) node[left]{$7$};;
\draw[fill=black] (0,8) circle (2pt) node[left]{$8$};;

\draw[fill=black] (1,0) circle (2pt) node[below]{$1$};;
\draw[fill=black] (2,0) circle (2pt) node[below]{$2$};;
\draw[fill=black] (3,0) circle (2pt) node[below]{$3$};;
\draw[fill=black] (4,0) circle (2pt) node[below]{$4$};;
\draw[fill=black] (5,0) circle (2pt) node[below]{$5$};;
\draw[fill=black] (6,0) circle (2pt) node[below]{$6$};;
\draw[fill=black] (7,0) circle (2pt) node[below]{$7$};;
\draw[fill=black] (8,0) circle (2pt) node[below]{$8$};;

\draw[fill=black] (2,4.5) circle (2pt);

\draw[fill=black] (4,1.5) circle (2pt);


\draw[fill=white] (2,3) circle (4pt);
\draw[fill=white] (2,4) circle (4pt);
\draw[fill=white] (2,5) circle (4pt);
\draw[fill=white] (2,6) circle (4pt);


\draw[fill=black] (3,3) circle (4pt);
\draw[fill=white] (3,4) circle (4pt);
\draw[fill=white] (3,5) circle (4pt);
\draw[fill=white] (3,6) circle (4pt);

\draw[fill=white] (4,3) circle (4pt);
\draw[fill=black] (4,4) circle (4pt);
\draw[fill=white] (4,5) circle (4pt);
\draw[fill=white] (4,6) circle (4pt);

\draw[fill=white] (5,3) circle (4pt);
\draw[fill=white] (5,4) circle (4pt);
\draw[fill=black] (5,5) circle (4pt);
\draw[fill=white] (5,6) circle (4pt);

\draw[fill=white] (6,3) circle (4pt);
\draw[fill=white] (6,4) circle (4pt);
\draw[fill=white] (6,5) circle (4pt);
\draw[fill=black] (6,6) circle (4pt);

\path[thin] (1.9,4.5) edge[->,bend left=22,looseness=0.8] (1.9,7.5);
\draw (1.9,7.5) node[above] {$T^{2,\frac92}$};

\draw (4,1.5) node[right] {$T^{4,\frac32}$};

\end{tikzpicture}}
\end{center}
\caption{Plot of the decomposition of the tensor product $T^{4,\frac{3}{4}}\otimes T^{2,\frac{9}{2}}$ into irreducible representations. The big nodes indicate the irreducible summands, the big black nodes the ones that contain the trivial $K_*$-type upon restricting.}\label{figure:decomp}
\end{figure}

Given $d,e\in\bbN$ and $-\ell\le p,q\le\ell$ we write $\zeta(d,p)=(\ell_{1},\ell_{2})$, $\zeta(e,q)=(m_{1},m_{2})$. Then we have
$$
\left(\Phi^{\ell}_{d}(a_{t})\left(\Phi^{\ell}_{e}\right)(a_{t})^{*}\right)_{p,q}=
\tr\left(\Phi^{\ell}_{\zeta(d,p)}(a_{t})\left(\Phi^{\ell}_{\zeta(e,q)}(a_{t})\right)^{*}\right)
=\sum_{j,j_{1},j_{2},i_{1},i_{2}}\left(C^{\ell_{1},\ell_{2},\ell}_{j_{1},j_{2},j}C^{m_{1},m_{2},\ell}_{i_{1},i_{2},j}\right)^{2}e^{i(j_{2}-j_{1}+i_{1}-i_{2})t},
$$
where the sum is taken over
$$|j|\le\ell,\quad|j_{1}|\le\ell_{1},\quad|j_{2}|\le\ell_{2},\quad|i_{1}|\le m_{1},\quad|i_{2}|\le m_{2},$$
satisfying $j_{1}+j_{2}=i_{1}+i_{2}=j$.
This equals
\begin{equation}\label{eq:mc erik}
\left(\Phi^{\ell}_{d}(a_{t})\left(\Phi^{\ell}_{e}(a_{t})\right)^{*}\right)_{p,q}=\sum_{|s|\le\min(\ell_{1}+m_{1},\ell_{2}+m_{2})}d^{\ell}_{\ell_{1},\ell_{2},m_{1},m_{2},s}e^{ist}
\end{equation}
where
\begin{equation}\label{eq:coef d}
d^{\ell}_{\ell_{1},\ell_{2},m_{1},m_{2},s}=\sum_{j,j_{1},j_{2},i_{1},i_{2}}\left(C^{\ell_{1},\ell_{2},\ell}_{j_{1},j_{2},j}C^{m_{1},m_{2},\ell}_{i_{1},i_{2},j}\right)^{2},
\end{equation}
where the sum is taken over
$$|j|\le\ell,\quad|j_{1}|\le\ell_{1},\quad|j_{2}|\le\ell_{2},\quad|i_{1}|\le m_{1},\quad|i_{2}|\le m_{2},$$
satisfying $j_{1}+j_{2}=i_{1}+i_{2}=j$ and $j_{2}-j_{1}+i_{1}-i_{2}=s$.
Since $U_{n}(\cos t)=e^{-int}+e^{-i(n-2)t}+\cdots+e^{int}$, it follows from (\ref{eq:mc erik}) and Lemma \ref{lem: erik} that we have the following summation result.
\begin{cor}\label{cor:independenceofs}
Let $|s|\le\max(|\ell_{1}-m_{1}|,|\ell_{2}-m_{2}|)$. Then $d^{\ell}_{\ell_{1},\ell_{2},m_{1},m_{2},s}$ is independent of $s$.
\end{cor}
Note that the sum in (\ref{eq:coef d}) is a double sum of four Clebsch-Gordan coefficients, which in general are $_{\,3}F_{2}$-series \cite{VilenkinKlimyk3vol}.

We now turn to the case $\ell_{1}=\frac{\ell+p}{2},\ell_{2}=\frac{\ell-p}{2},m_{1}=\frac{\ell+q}{2},m_{2}=\frac{\ell-q}{2}$. Because of Lemma \ref{lemma:symmetries} the next theorem gives an explicit expression for the weight matrix.

\begin{theorem}\label{superweight}
Let $q-p\le0$ and $q+p\le0$. For $n=0,\ldots,\ell+q$ there are coefficients $c_{n}^{\ell}(p,q)\in\bbQ_{>0}$ such that
\begin{equation}\label{eqn:c_s}
\left(V^{\ell}(a_{t})\right)_{p,q}=\sin^{2}(t)\sum_{n=0}^{\ell+q}c_{n}^{\ell}(p,q)U_{2\ell+p+q-2n}(\cos(t))
\end{equation}
where the coefficients are given by
\begin{equation}\label{supercoef} c^{\ell}_{n}(p,q)=\frac{2\ell+1}{\ell+p+1}\frac{(\ell-q)!(\ell+q)!}{(2\ell)!}\frac{(p-\ell)_{\ell+q-n}}{(\ell+p+2)_{\ell+q-n}}(-1)^{\ell+q-n}\frac{(2\ell+2-n)_{n}}{n!}.
\end{equation}
\end{theorem}
By $c_{n+p+q}(p,q)=c_{n}(-q,-p)$ the expansion (\ref{eqn:c_s}) with (\ref{supercoef}) remains valid for $q\le p$.
\begin{proof}
Since $\min(\ell_{1}+m_{1},\ell_{2}+m_{2})=\ell+\frac{p+q}{2}$ and $\max(|\ell_{1}-m_{1}|,|\ell_{2}-m_{2}|)=\frac{p-q}{2}$ in this case we find the expansion of the form as stated in (\ref{eqn:c_s}). It remains to calculate the coefficients. Specializing (\ref{eq:mc erik}) and writing
\begin{equation}\label{eqn:CG}
\left(C^{(\ell+m)/2,(\ell-m)/2,\ell}_{j_{1},j_{2},j}\right)^{2}
=\delta_{j,j_{1}+j_{2}}\frac{\binom{\ell+m}{j_{1}+(\ell+m)/2}\binom{\ell-m}{j_{2}+(\ell-m)/2}}{\binom{2\ell}{\ell-j}}.
\end{equation}
we find
\begin{equation}\label{eq:weight4}
\begin{split}
v^\ell_{pq}(\cos t) \, = \, &
\sum_{j=-\frac{\ell+p}{2}}^{\frac{\ell+p}{2}} 
\sum_{i=-\frac{\ell+q}{2}}^{\frac{\ell+q}{2}}
F_{ij}^\ell(p,q) \exp(i(-2(i+j)t)) = \sum_{r=-(\ell+\frac{p+q}{2})}^{\ell+\frac{p+q}{2}} 
\left( \sum_{i=\max(-\frac{\ell+q}{2}, r-\frac{\ell+p}{2})}^{\min(\frac{\ell+q}{2}, r+\frac{\ell+p}{2})} F^\ell_{i,r-i}(p,q) \right) e^{-2irt},
\end{split}
\end{equation}
with
\begin{equation}\label{eq:defFij}
F_{ij}^\ell(p,q)\, = \, 
\binom{\ell+p}{j+(\ell+p)/2}\, \binom{\ell+q}{i+(\ell+q)/2}\,\\
\sum_{k=\max(-j-\frac{\ell-p}{2},i-\frac{\ell-q}{2})}^{\min(-j+\frac{\ell-p}{2},i+\frac{\ell-q}{2})}
\frac{\binom{\ell-p}{-k-j+(\ell-p)/2}\binom{\ell-q}{k-i+(\ell-q)/2}}{\binom{2\ell}{\ell-k}^2}.
\end{equation}
From this we can obtain the explicit expression of $v^{\ell}(a_{t})_{p,q}$ in Chebyshev polynomials. The details are presented in Appendix A.
\end{proof}

\begin{prop}\label{prop:commutants}
The commutant
$$\{V^{\ell}(a):a\in A_{*}\}':=\\
\{Y\in\End(H^{\ell}):V^{\ell}(a)Y=YV^{\ell}(a)
\quad\mbox{for all }a\in A_{*}\}=\{v^{\ell}(a):a\in A_{*}\}',$$
is spanned by the matrices
$I$ and $J$.
\end{prop}

\begin{proof}
By Proposition \ref{prop:VandJcomm} we have $J\in\{V^{\ell}(a):a\in A_{*}\}'$. It suffices to show that the commutant contains no other elements than those spanned by $I$ and $J$.

Let $v^{\ell}(a_{t})=\sum_{n=0}^{2\ell}U_{n}(\cos t) A_{n}$, with $A_{n}\in\Mat_{2\ell+1}(\bbC)$ by Theorem \ref{superweight}. Then for $B$ in the commutant it is necessary and sufficient that $A_{n}B=BA_{n}$ for all $n$. First, put $C=A_{2\ell}$. Then by Theorem \ref{superweight} $C_{p,q}=\binom{2\ell}{\ell+p}^{-1}\delta_{p,-q}$. The equation $BC=CB$ leads to $B_{p,q}=(C^{-1}BC)_{p,q}=\frac{C_{q,-q}}{C_{p,-p}}B_{-p,-q}=\frac{C_{q,-q}C_{-q,q}}{C_{p,-p}C_{-p,p}}B_{p,q}$ by iteration. Since $C_{q,-q}= C_{p,-p}$ if and only if $p=q$ or $p=-q$ we find $B_{p,q}=0$ for $p\ne q$ or $p\ne-q$. Moreover, $B_{p,p}=B_{-p,-p}$ and $B_{p,-p}=B_{-p,p}$.

Secondly, put $C'=A_{2\ell-1}$. Then, by Theorem \ref{superweight}, we have
$$C'_{p,q}=\delta_{|p+q|,1}(2\ell+1)\binom{2\ell}{\ell-p}^{-1}\binom{2\ell}{\ell+q}^{-1},$$
so the non-zero entries are different up to the symmetries $C'_{p,q}=C'_{-p,-q}=C'_{q,p}=C'_{-q,-p}$. 
Now $BC'=C'B$ implies by the previous result $B_{p,p}C'_{p,q}+B_{p,-p}C'_{-p,q}=C'_{p,q}B_{q,q}+C'_{p,-q}B_{q,q}$. Take $q=1-p$ to find $B_{p,p}=B_{1-p,1-p}=B_{p-1,p-1}$ unless $p=0$ or $p=1$, and take $q=p-1$ to find $B_{p,-p}=B_{1-p,p-1}=B_{p-1,1-p}$ unless $p=0$ or $p=1$. In particular, for $\ell\in\frac{1}{2}+\bbN$ this proves the result. In case $\ell\in\bbN$ we obtain one more equation: $B_{0,0}=B_{1,1}+B_{1,-1}$. This shows that $B$ is in the span of $I$ and $J$.
\end{proof}
The matrix $J$ has eigenvalues $\pm1$ and two eigenspaces $H^{\ell}_{-}$
and $H^{\ell}_{+}$. The dimensions are $\lfloor\ell+1/2\rfloor$ and
$\lceil\ell+1/2\rceil$. A choice of (ordered) bases of the eigenspaces is given
by
\begin{equation}\label{eqn:basis}
\{\psi^{\ell}_{j}-\psi^{\ell}_{-j}:-\ell\le
j<0,\ell-j\in\bbZ\}\quad\mbox{and}\quad\{\psi^{\ell}_{j}+\psi^{\ell}_{-j}:0\le
j\le\ell,\ell-j\in\bbZ\}.
\end{equation}
Let $Y_{\ell}$ be the matrix whose columns are the normalized basis
vectors of (\ref{eqn:basis}). Conjugating $V^{\ell}$ with $Y_{\ell}$
yields a matrix with two blocks, one block of size
$\lceil\ell+1/2\rceil\times\lceil\ell+1/2\rceil$ and one of size
$\lfloor\ell+1/2\rfloor\times\lfloor\ell+1/2\rfloor$.

\begin{cor}\label{cor:weight_V_splits}
The family $(Q_{d}^{\ell})_{d\ge0}$ and the weight $V^{\ell}$ are conjugate to a
family and a weight in block form. More precisely
$$Y_{\ell}^{-1}Q_{d}^{\ell}(a_{t})Y_{\ell}=\left(\begin{array}{cc}Q_{d,-}^{\ell}(a_{t})&0\\0&Q_{d,+}^{\ell}(a_{t})\end{array}\right),
\  Y_{\ell}^{-1}V^{\ell}(a_{t})Y_{\ell}=\left(\begin{array}{cc}V_{-}^{\ell}(a_{t})&0\\0&V_{+}^{\ell}(a_{t})\end{array}\right).$$
The families $(Q_{d,\pm}^{\ell})_{d\ge0}$ are orthogonal with respect to the
weight $V_{\pm}^{\ell}$. Moreover, there is no further reduction possible.
\end{cor}

\begin{proof}
The functions $Q_{d}^{\ell}$ can be conjugated by $Y_{\ell}$. Since
the $Q_{d}^{\ell}$ commute with $J$ we see that the
$Y_{\ell}^{-1}Q_{d}^{\ell}Y_{\ell}$ has the same block structure as
$Y_{\ell}^{-1}V^{\ell}Y_{\ell}$. The blocks of the
$Y_{\ell}^{-1}Q_{d}^{\ell}Y_{\ell}$ are orthogonal with respect to the
corresponding block of $Y_{\ell}^{-1}V^{\ell}Y_{\ell}$. The polynomials $Q_{d,-}^{\ell}$
take their values in the $(-1)$-eigenspace $H^{\ell}_{-}$ of $J$, the polynomials
$Q_{d,+}^{\ell}$ in the $(+1)$-eigenspace $H^{\ell}_{+}$ of $J$. The dimensions are
$\lfloor\ell+1/2\rfloor$ and $\lceil\ell+1/2\rceil$ respectively.

A further reduction would require an element in the commutant $\{V^{\ell}(a):a\in A_{*}\}'$ not in the span of $I$ and $J$. This is not possible by Proposition \ref{prop:commutants}.
\end{proof}

The entries of the weight $v^{\ell}$ with the Chebyshev polynomials of the highest degree $2\ell$ occur only on the antidiagonal by Theorem \ref{superweight}. This shows that the determinant of $v^{\ell}(a_t)$ is a polynomial 
in $\cos t$ of degree $2\ell(2\ell+1)$ with leading coefficient $(-1)^{\ell(2\ell+1)}\prod_{p=-\ell}^{\ell}c_{0}(p,-p)2^{2\ell}\ne0$. Hence $v^{\ell}$ is invertible on $A_{*}$ away from the zeros of its determinant, of which there are only finitely many. We have proved the following proposition which should be compared to Proposition \ref{prop:invertible}.
\begin{prop}\label{prop:invertible2} The full spherical function $\Phi_{0}^{\ell}$ is invertible on $A_{*}$ except for a finite set.
\end{prop}

In particular, $Q^\ell_d$ is well-defined in Definition \ref{def:fullsf}, except
for a finite set. Since $Q^\ell_d$ is polynomial, it is well-defined on $A$. 

Mathematica calculations lead to the following conjecture.
\begin{conjecture}\label{conj:inv}
$\det(v^{\ell}(a_{t}))=(1-\cos^2 t)^{\ell(2\ell+1)}\prod_{p=-\ell}^{\ell}(2^{2\ell} c_0^{\ell}(p,-p))$.
\end{conjecture}
Conjecture \ref{conj:inv} is supported by Koornwinder \cite{Koornwinder85}*{Prop. 3.2}, see Proposition \ref{prop:invertible}.
Conjecture \ref{conj:inv} has been verified for $\ell \leq 16$. 


\section{The matrix orthogonal polynomials associated to $(\SU(2)\times\SU(2),\diag)$}
\label{sec:orthogonal_polynomials}\label{6}
The main goal of Sections \ref{sec:recurrence_relation}, \ref{sec:restricted_sf} and
\ref{sec:symmetries_of_W} was to study the properties of the matrix
valued spherical functions of any $K$-type associated to the pair
$(\SU(2)\times\SU(2),\diag)$. These functions, introduced in Definition \ref{def:sf},
are the building blocks of the full spherical functions described in
Definition \ref{def:fullsf}. We have exploited the fact that the spherical functions
diagonalize when restricted to the subgroup $A$. This allows us to identify
each spherical function with a row vector and arrange them in a square matrix.

The goal of this section is to translate the properties of the full spherical functions
obtained in the previous sections at the group level to the corresponding family of matrix valued
orthogonal polynomials.

\subsection{Matrix valued orthogonal polynomials} Let $W$ be a complex $N\times N$ matrix 
valued integrable function on the interval $(a,b)$ such that
$W$ is positive definite almost everywhere and with finite 
moments of all orders. Let $\text{\rm{Mat}}_N(\bbC)$
be the algebra of all $N\times N$ complex matrices. The algebra over $\mathbb{C}$  of all polynomials
in the indeterminate $x$ with coefficients in $\text{\rm{Mat}}_N(\bbC)$
denoted by $\text{\rm{Mat}}_N(\bbC)[x]$.
Let $\langle\cdot,\cdot\rangle$ be the following 
Hermitian sesquilinear form in the linear space $\text{\rm{Mat}}_N(\bbC)[x]$:
\begin{equation}
\label{eq:HermitianForm}
\langle P,Q \rangle=\int_a^b P(x)W(x)Q(x)^*dx.
\end{equation}
The following properties are satisfied:
\begin{enumerate}\setlength{\itemsep}{1.5mm} 
\item $\langle aP+bQ,R\rangle=a\langle P,R\rangle+b\langle Q,R\rangle$, for all $P,Q,R\in \text{\rm{Mat}}_N(\bbC)[x]$, $a,b\in \mathbb{C}$,
\item $\langle TP,Q\rangle=T\langle P,Q\rangle$, for all  $P,Q\in \text{\rm{Mat}}_N(\bbC)[x]$, $T\in \text{\rm{Mat}}_N(\bbC)$,
\item $\langle P,Q\rangle^*=\langle Q,P\rangle$, for all $P,Q\in \text{\rm{Mat}}_N(\bbC)[x]$,
\item $\langle P,P\rangle\geq0$ for all $P\in \text{\rm{Mat}}_N(\bbC)[x]$. Moreover if $\langle P, P\rangle=0$ then $P=0$.
\end{enumerate}

Given a weight matrix $W$ one constructs a sequence of matrix valued orthogonal polynomials,
that is a sequence $\{R_n\}_{n\geq0}$, where $R_n$ is a polynomial of degree $n$
with nonsingular leading coefficient and $\langle R_n,R_m\rangle=0$ if $n\neq m$.

It is worth noting that there exists a unique sequence of monic orthogonal
polynomials $\{P_n\}_{n\geq 0}$ in $\text{\rm{Mat}}_N(\bbC)[x]$. Any other
sequence of $\{R_n\}_{n\geq 0}$ of orthogonal polynomials in $\text{\rm{Mat}}_N(\bbC)[x]$ is
of the form $R_n(x)=A_nP_n(x)$ for some $A_n\in \operatorname{GL}_N(\mathbb{C})$.

By following a well-known argument, see for instance \cite{Krein1}, \cite{Krein2}, one
shows that the monic orthogonal polynomials $\{P_n\}_{n\geq 0}$ satisfy a three-term recurrence relation
$$xP_n(x)=P_{n+1}(x)+B_{n}(x)P_n(x)+C_nP_{n-1}(x), \quad n\geq 0,$$
where $Q_{-1}=0$ and $B_n$, $C_n$ are matrices depending on $n$ and not on $x$.

There is a notion of similarity between two weight matrices that was pointed out in \cite{DG2}.
The weights $W$ and $\widetilde W$ are said to be similar if there exists a nonsingular
matrix $M$, which does not depend on $x$, such that $\widetilde W(x)=MW(x)M^*$ for all $x\in(a,b)$.
\begin{prop} \label{prop:polynomials_split}
Let  $\{R_{n,1}\}_{n\geq 0}$ be a sequence of orthogonal
polynomials with respect to $W$ and $M\in GL_N(\bbC)$. Then the sequence $\{R_{n,2}(x)=R_{n,1}(x)M^{-1}\}_{n\geq 0}$
is orthogonal with respect to $\widetilde W=MWM^*$. Moreover, if $\{P_{n,1}\}$ is the
sequence of monic orthogonal polynomials orthogonal with respect to $W$ then $\{P_{n,2}(x)=MP_{n,1}(x)M^{-1}\}$ is the
sequence of monic orthogonal polynomials with respect to $\widetilde W$.
\end{prop}
\begin{proof}
It follows directly by observing that
\begin{align*}
\int R_{n,2}(x) \widetilde W(x) R_{m,2}(x)^* dx &= \int  R_{n,1}(x)M^{-1} \widetilde W(x)(M^{-1})^*  R_{m,1}(x)^* dx\\
&=\int  R_{n,1}(x) W(x) R_{m,1}(x)^* dx =0,\quad \text{ if }n\neq m.
\end{align*}
The second statement follows by looking at the leading coefficient of $P_{n,2}$ and the unicity
of the sequence of monic orthogonal polynomials with respect to $\widetilde W$.
\end{proof}
A weight matrix $W$ reduces to a smaller size if there exists a matrix $M$ such that
$$W(x)= M \begin{pmatrix} W_1(x) & 0 \\ 0 & W_2(x) \end{pmatrix} M^*, \quad \text{for all }x\in(a,b),$$
where $W_1$ and $W_2$ are matrix weights of smaller size. In this case the monic polynomials
 $\{P_n\}_{n\geq 0}$ with respect to the weight $W$ are given by
$$P_n(x)=M\begin{pmatrix} P_{n,1}(x) & 0 \\ 0 & P_{n,2}(x) \end{pmatrix}M^{-1}, \quad n\geq 0,$$
where $\{P_{n,1}\}_{n\geq 0}$ and $\{P_{n,2}\}_{n\geq 0}$ are the monic orthogonal polynomials with
respect to $W_1$ and $W_2$ respectively.

\subsection{Polynomials associated to $\SU(2)\times\SU(2)$}

In the rest of the paper we will be concerned with the properties of the
matrix orthogonal polynomials $Q_d$. For this purpose
we find convenient to introduce a new labeling in the rows and columns of the
weight $V$. More precisely for any 
$\ell\in \frac12\mathbb{Z}$ let $W$ be the $(2\ell+1)\times(2\ell+1)$ matrix given
by
\begin{equation}
\sqrt{1-x^2}\, W(x)_{n,m}=V(a_{\arccos{x}})_{-\ell+n,-\ell+m}, \quad n,m\in\{0,1,\ldots,2\ell\}.
\end{equation}
It then follows from Theorem \ref{superweight} that
\begin{equation}
\label{def:matrix_W}
W(x)_{n,m}=(1-x)^{\frac12}(1+x)^{\frac12}\frac{(2\ell+1)}{n+1}\frac{(2\ell-m)!m!}{(2\ell)!}\\
\sum_{t=0}^m (-1)^{m-t} \frac{(n-2\ell)_{m-t}}{(n+2)_{m-t}}
\frac{(2\ell+2-t)_t}{t!}U_{n+m-2t}(x),
\end{equation}
if $n\leq m$ and $W(x)_{n,m}=W(x)_{m,n}$ otherwise. 

We also consider the sequence of monic polynomials $\{P_d\}_{d\geq 0}$ given by
\begin{equation}\label{eq:def_polynomials_Pd}
P_d(x)_{n,m}=\Upsilon_d^{-1} Q_d(a_{\arccos{x}})_{-\ell+n,-\ell+m}, \quad n,m\in\{0,1,\ldots,2\ell\},
\end{equation}
where $\Upsilon_d$ is the leading coefficient of the polynomial $Q_d(a_{\arccos{x}})$,
which is non-singular by Theorem \ref{thm:three_term_for_Qd}. 
Now we can rewrite the results on Section \ref{sec:symmetries_of_W} in terms of the weight $W$
and the polynomials $P_d$.
\begin{cor}
The sequence of matrix
polynomials $\{P_{d}(x)\}_{d>0}$ is orthogonal with respect to the
matrix valued inner product
$$\langle P, Q \rangle = \int_{-1}^1 P(x)W(x)Q(x)^*dx.$$
\end{cor}

Theorem \ref{thm:three_term_for_Qd} states that there is a three term recurrence relation
defining the matrix polynomials $Q_d$. These polynomials are functions on the
group $A$. We can use \eqref{eq:def_polynomials_Pd} to derive a three term recurrence relation
for the polynomials $P_d$.
\begin{cor}\label{thm:three_term_for_Pd}
 For any $\ell\in\frac{1}{2}\mathbb{N}$ the matrix valued orthogonal polynomials
$P_d$, are defined by the following three term recurrence relation
\begin{equation}
xP_{d}(x)= P_{d+1}(x)+ \Upsilon_d^{-1} B_{d} \Upsilon_d P_{d}(x)+ 
\Upsilon_d^{-1}C_{d}\Upsilon_{d-1}P_{d-1}(x),
\end{equation}
where the matrices $A_d$, $B_d$ and $C_d$ are given in 
Theorem \ref{thm:three_term_for_Qd} and taking into account the relabeling
as in the beginning of this subsection. 
\end{cor}

\subsection{Symmetries of the weight and the matrix polynomials}
In this section we shall use the symmetries satisfied by the full
spherical functions to derive symmetry properties for the
matrix weight $W$ and the polynomials $P_d$.

For any $n\in \mathbb{N}$, let $I_n$ be the $n\times n$ identity
matrix and let $J_n$ and $F_n$ be the following $n\times n$ matrices
\begin{equation}\label{eq:matrices_Jn_Fn}
J_n=\sum_{i=0}^{n-1} E_{i,n-1-i}, \quad F_n= \sum_{i=0}^{n-1} (-1)^i E_{i,i}.
\end{equation}
For any $n\times n$ matrix $X$ the transpose $X^t$ is defined
by $(X^t)_{ij} = X_{ji}$ (reflection in the diagonal)
and we define the reflection in the antidiagonal
by $(X^d)_{ij}= X_{n-j,n-i}$. Note that taking transpose and taking
antidiagonal transpose commute, and that
\[
(X^t)^d\, = \, (X^d)^t \, = \, X^{dt}\, = \, J_nXJ_n.
\]
Moreover, $(XZ)^d= Z^dX^d$ for arbitrary matrices $X$ and $Z$. 
We also need to consider the $(2\ell+1)\times(2\ell+1)$ matrix $Y$ defined by
\begin{equation}
\label{eq:definition_M}
\begin{split}
Y&=\frac{1}{\sqrt{2}}\begin{pmatrix} I_{\ell+\frac12} & J_{\ell+\frac12} \\
-J_{\ell+\frac12} & I_{\ell+\frac12} \end{pmatrix},\text{ if }\ell
=\frac{2n+1}{2},\quad n\in \mathbb{N},\\
Y&=\frac{1}{\sqrt{2}}\begin{pmatrix} I_{\ell}& 0 & J_{\ell} \\ 0 & \sqrt{2} & 0 \\
-J_{\ell}& 0 & I_{\ell} \end{pmatrix},\text{ if }\ell\in \mathbb{N}.
\end{split}
\end{equation}

\begin{prop}\label{prop:W_is_symmetric_and_(2)}
The weight matrix $W(x)$ satisfies the
following symmetries
\begin{enumerate}
\item $W(x)^t=W(x)$ and $W(x)^d=W(x)$ for all $x\in[-1,1]$. Thus
$$J_{2\ell+1}W(x)J_{2\ell+1}=W(x),$$
for all $x\in[-1,1]$.
\item $W(-x)=F_{2\ell+1}W(x)F_{2\ell+1}$ for all $x\in[-1,1]$.
\end{enumerate}
Here $F_{2\ell+1}$ is the $(2\ell+1)\times(2\ell+1)$ matrix given in \eqref{eq:matrices_Jn_Fn}.
\end{prop}
\begin{proof}
The symmetry properties of $W$ in (1) follow directly from Lemma \ref{lemma:symmetries}.

The proof of (2) follows from \eqref{def:matrix_W} by
using the fact that $U_{n}(-x)=(-1)^nU_{n}(x)$ for any Chebyshev polynomial of the second
kind $U_n(x)$, so that $W(-x)_{n,m} = (-1)^{n+m} W_{n,m}(x)$. 
\end{proof}
The weight matrix $W$ can be conjugated into a $2\times2$ block  diagonal matrix.
In Corollary \ref{cor:weight_V_splits} we have pointed out this phenomenon for the weight $V$.
The following theorem translates Corollary \ref{cor:weight_V_splits} to the weight matrix $W$.
\begin{theorem}\label{thm:W_splits}
For any $\ell\in\frac12\mathbb{N}$, the matrix $W$ satisfies
$$\widetilde W(x)=YW(x)Y^t=\begin{pmatrix} W_1(x) & 0 \\  0 & W_2(x) \end{pmatrix},$$
where $Y$ is matrix given by \eqref{eq:definition_M}. Moreover if $\{P_{d,1}\}_{d\geq 0}$
(resp. $\{P_{d,2}\}_{d\geq 0}$) is a sequence of monic matrix orthogonal polynomials with respect to the
weight $W_1(x)$ (resp. $W_2(x)$), then
\begin{equation}\label{eq:definition_polynomials_Pn_tilde}
\widetilde P_d(x)= \begin{pmatrix} P_{d,1}(x) & 0 \\ 0 & P_{d,2} \end{pmatrix}, \quad d\geq 0,
\end{equation}
is a sequence of matrix orthogonal polynomials with respect to $\widetilde W$. There is no
further reduction.
\end{theorem}
The case $\ell=(2n+1)/2$, $n\in\mathbb{N}$, leads to weights matrices $W$ of
even size. In this case $W$ splits into two blocks of size $\ell+\frac12$.
In Corollary \ref{thm:W_splits_even_case} we prove that these two blocks are equivalent, hence the
corresponding matrix orthogonal polynomials are equivalent.

It follows from Proposition \ref{prop:W_is_symmetric_and_(2)} (1) that there
exist $(n+1)\times(n+1)$ matrices $A(x)$ and $B(x)$ such that
$A(x)^t=A(x)$ and 
\begin{equation}\label{eq:W_blocks_A_B}
W(x)=\begin{pmatrix} A(x) & B(x) \\ B(x)^{dt} & A(x)^{dt} \end{pmatrix},
\end{equation}
for all $x\in[-1,1]$.

\begin{cor}\label{thm:W_splits_even_case}
Let $\ell=(2n+1)/2$, $n\in\mathbb{Z}$. Then
$$YW(x)Y^t=\begin{pmatrix} W_1(x) & 0 \\ 0 & W_2(x) \end{pmatrix},$$
where
$$W_1(x)=A(x)+ B(x)J_{n+1}, \quad W_2(x)=J_{n+1}F_{n+1}W_1(-x)F_{n+1}J_{n+1}.$$
Here $A(x)$ and $B(x)$ are the matrices described in \eqref{def:matrix_W} and \eqref{eq:W_blocks_A_B}.
Moreover, if $\{P_{d,1}\}_{d\geq 0}$ is the sequence of monic orthogonal polynomials with
respect to $W_1(x)$ then
\begin{equation}\label{eq:equation_Pd2}
P_{d,2}(x)=(-1)^d J_{n+1}F_{n+1}P_{d,1}(-x)F_{n+1}J_{n+1},
\end{equation}
is the sequence of monic orthogonal polynomials with respect to $W_2(x)$.
\end{cor}
\begin{proof}
In this proof we will drop the subindex in the matrices $J_{n+1}$, $F_{n+1}$ and we will
use $J$ and $F$ instead. It is a straightforward
calculation that for $(n+1)\times (n+1)$-matrices $A$, $B$ $C$ and
$D$ the following holds
\begin{align*}
Y\begin{pmatrix} A & B \\ C & D
  \end{pmatrix} Y^t \, = \,
\frac12 \,
\begin{pmatrix} A + D^{dt} + J(C+B^{dt}) & B-C^{dt} + (D^{dt}-A)J \\
J(D^{dt}-A) + C-B^{dt} & D+A^{dt} - (C+B^{dt})J
\end{pmatrix}
\end{align*}
In particular for the weight function $W$ we get
\begin{align*}
YW(x)Y^t&=Y \begin{pmatrix} A(x) & B(x) \\ B(x)^{dt} & A(x)^{dt}
  \end{pmatrix} Y^t \,
= \,
\begin{pmatrix} A(x) + B(x)J & 0 \\
0 & J(A(x) - B(x)J)J
\end{pmatrix}
\end{align*}
This proves that
$$W_1(x)=A(x) + B(x)J, \quad W_2(x)=J(A(x) - B(x)J)J.$$
It follows from Proposition \ref{prop:W_is_symmetric_and_(2)} (2) that
$A(-x)=FA(x)F$ and $B(-x)=FB(x)F$. Therefore we have
$$JFW_1(-x)FJ=JFA(-x)FJ+FJB(-x)FJ=JA(x)J-JB(x)=W_2(x).$$
This proves the first assertion of the theorem.

The last statement follows from Proposition \ref{prop:polynomials_split}
\end{proof}


\section{Matrix valued differential operators}\label{7}
%

In the study of matrix-valued orthogonal polynomials an important ingredient is
the study of differential operators which have these matrix-valued 
orthogonal polynomials as eigenfunctions. In this section
we discuss some of the differential operators that have the 
matrix-valued orthogonal polynomials of the previous section 
as eigenfunctions. The calculations rest on the explicit form of
the weight function \eqref{def:matrix_W}.

\subsection{Symmetric differential operators}
We consider right hand side differential operators
\begin{equation}\label{eq:D_general}
D=\sum_{i=0}^s\partial^i F_i(x), \quad \partial=\frac{d}{dx},
\end{equation}
in such a way that the action of $D$ on the polynomial $P(x)$ is
$$PD=\sum_{i=0}^s \partial^i(P)(x)F_i(x).$$
In \cite{GT}*{Propositions 2.6 and 2.7} one can find a proof of the following proposition.
\begin{prop} \label{prop:eigenvalue} Let $W=W(x)$ be a weight matrix of size $N$ and let $\{P_n\}_{n\geq 0}$
be the sequence of monic orthogonal polynomials in $\text{\rm{Mat}}_N(\bbC)[x]$. If $D$ is a
right hand side ordinary differential operator as in \eqref{eq:D_general} of order $s$ such that
$$P_nD=\Lambda_nP_n, \quad \text{for all }n\geq 0,$$
with $\Lambda_n\in A$, then
$$F_i=F_i(x)=\sum_{j=0}^i x^jF_j^i,\quad F_j^i \in \text{\rm{Mat}}_N(\bbC),$$
is a polynomial of degree less than or equal to $i$. Moreover $D$ is determined by
the sequence $\{\Lambda_n\}_{n\geq 0}$ and
$$\Lambda_n=\sum_{i=0}^s[n]_iF_i^i(D), \quad \text{for all } n\geq 0,$$
where $[n]_i=n(n-1)\cdots(n-i+1)$, $[n]_0=1$.
\end{prop}
We consider the following algebra of right hand side differential operators with
 coefficients in $\text{\rm{Mat}}_N(\bbC)[x]$.
\begin{equation*}
\mathcal{D}=\{D=\sum_i \partial^i F_i: F_i\in \text{\rm{Mat}}_N(\bbC)[x],\, \deg F_i\leq i \}.
\end{equation*}
Given any sequence of matrix valued orthogonal polynomials $\{R_n\}_{n\geq0}$ 
with respect to $W$, we define
\begin{equation*}
\mathcal{D}(W)=\{D\in \mathcal{D}\, :\, R_nD=\Gamma_n(D)R_n, \, \Gamma_n(D)\in \text{\rm{Mat}}_N(\bbC),\,\text{for all }n\geq0 \}.
\end{equation*}
We observe that the definition of $\mathcal{D}(W)$ does not depend on the sequence of orthogonal
polynomials $\{R_n\}_{n\geq 0}$.

\begin{remark}\label{rmk:eigenvalues_are_representation} The mapping
$D\mapsto \Gamma_n(D)$ is a representation of  $\mathcal{D}(W)$ in $\mathbb{C}^{N}$ for each $n\geq 0$.
Moreover the family of representations $\{\Gamma_n\}_{n\geq0}$ separates the
points of $\mathcal{D}(W)$. Note that $\mathcal{D}(W)$ is an algebra.
\end{remark}
\begin{definition}\label{def:symmetric}
A differential operator $D\in \mathcal{D}$ is said to be symmetric if $\langle PD,Q\rangle =\langle P,QD \rangle$
for all $P,Q\in \text{\rm{Mat}}_N(\bbC)[x]$.
\end{definition}
\begin{prop}[\cite{GT}]\label{prop:Dsymmetric_D(W)} If $D\in \mathcal{D}$ is symmetric then $D\in \mathcal{D}(W)$.
\end{prop}
The main theorem in \cite{GT} says that for any $D\in \mathcal{D}$ there exists
a unique differential operator $D^*\in\mathcal{D}(W)$, the adjoint of $D$,
such that $\langle PD,Q\rangle=\langle P,QD^*\rangle$ for all $P,Q\in \text{\rm{Mat}}_N(\bbC)[x]$.
The map $D\mapsto D^*$ is a $*$-operation in the algebra $\mathcal{D}(W)$ . Moreover we have
$\mathcal{D}(W)=\mathcal{S}(W)\oplus i \mathcal{S}(W)$,
where $\mathcal{S}(W)$ denotes the set of all symmetric operators. Therefore it suffices,
in order to determine all the algebra $\mathcal{D}(W)$, to determine the symmetric operators
$\mathcal{S}(W)$.

The condition of symmetry in Definition \ref{def:symmetric} can be translated into a set of
differential equations involving the weight $W$ and the coefficients of the differential operator $D$.
For differential operators of order $2$ this was proven in \cite{DG}*{Theorem 3.1}.
\begin{theorem}\label{thm:equations_symmetry}
Let $W(x)$ be a weight matrix supported on $(a,b)$. Let $D\in\mathcal{D}$ be the differential operator
$$D=\partial^2 F_2(x)+\partial F_1(x)+F^0_0,$$
Then $D$ is symmetric with respect to $W$ if and only if
\begin{align}
F_2W&=WF_2, \label{eq:conditions.symmetry1}\\
2(F_2W)' &=WF_1^*+F_1W,\label{eq:conditions.symmetry2}\\
(F_2W)''-(F_1W)'+F_0W&=WF_0^*,\label{eq:conditions.symmetry3}
\end{align}
with the boundary conditions
\begin{equation}
\label{eq:boundary.conditions}
\lim_{x\to a,b} F_2(x)W(x)=0,\quad \lim_{x\to a,b} (F_2(x)W(x))'-F_1(x)W(x)=0.
\end{equation}
\end{theorem}

\subsection{Matrix valued differential operators for the polynomials $P_n$}
As in the previous section, we will denote by $\{P_n\}_{n\geq 0}$ the sequence of monic orthogonal
polynomials with respect to the weight matrix $W$. We can write the 
weight as $W(x)=\rho(x)Z(x)$ where $\rho(x)=(1-x)^{\frac12}(1+x)^{\frac12}$ and
$Z(x)$ is the $(2\ell+1)\times (2\ell+1)$ matrix whose $(n,m)$-entry is given by
\begin{align}\label{eq:expression_Z}
Z(x)_{n,m}&=\frac{(2\ell+1)}{n+1}\frac{(2\ell-m)!m!}{(2\ell)!}
\sum_{t=0}^m (-1)^{m-t} \frac{(n-2\ell)_{m-t}}{(n+2)_{m-t}}\frac{(2\ell+2-t)_t}{t!}U_{n+m-2t}.\\
&=\sum_{t=0}^m c(n,m,t) U_{n+m-2t}(x), \nonumber
\end{align}
if $n\leq m$ and $Z(x)_{n,m}=Z(x)_{m,n}$ otherwise. 

Once we have an explicit expression for the weight matrix $W$ we can use the 
symmetry equations in Theorem \ref{thm:equations_symmetry} to find symmetric
differential operators. If we start with a generic second order differential
operator
$$D=\sum_{i=0}^2\partial^i F_i(x), \quad F_i(x)=\sum_{j=0}^i x^jF_j^i, 
\quad F_j^i\in \text{\rm{Mat}}_N(\bbC),$$
then the equations 
\eqref{eq:conditions.symmetry1}, \eqref{eq:conditions.symmetry2}
and \eqref{eq:conditions.symmetry3} lead to linear equations in the coefficients $F_j^i$.
It is easy to solve these equations for small values of $N$ using any software tool such as Maple.
We have used the general expressions for small values of $N$ to make an ansatz for the expressions
of a first order and a second order differential operator. Then we prove that these operators
are symmetric for all $N$ by showing that they satisfy the conditions in Theorem \ref{thm:equations_symmetry}.

In the following theorem we show the matrix polynomials $P_n$ satisfy a matrix 
valued first order differential equation. This, phenomenon 
which does not appear in the scalar, case has been
recently studied in the literature (see for instance \cite{CG1}, \cite{C1}).

\begin{theorem}
\label{conj:first_order_do}
Let $E$ be the first order matrix valued differential operator
\begin{equation*}
E=\left(\frac{d}{dx}\right)A_1(x)+A_0,
\end{equation*}
where the matrices $A_1(x)$ and $A_0$ are given by
\begin{align*}
A_1(x)&=\sum_{i=0}^{2\ell} \left(\frac{2\ell-i}{2\ell}\right) E_{i,i+1} -\sum_{i=0}^{2\ell} x\left(\frac{\ell-i}{\ell}\right) E_{ii} - \sum_{i=0}^{2\ell} \left(\frac{i}{2\ell}\right) E_{i,i-1},\\
A_0  &=\sum_{i=0}^{2\ell} \frac{(2\ell+2)(i-2\ell)}{2\ell} E_{ii}.
\end{align*}
Then $E$ is symmetric with respect to the weight $W$; hence $E\in D(W)$. Moreover for every integer $n\geq 0$,
$$P_n(x)E=\Lambda_n(E)P_n(x),$$
where
$$\Lambda_n(E)= \sum_{i=0}^{2\ell} \left(-\frac{n(\ell-i)}{\ell}+\frac{(2\ell+2)(i-2\ell)}{2\ell} \right)E_{ii}.$$
\end{theorem}
\begin{proof}
The proof of the theorem is performed by showing that the
differential operator $E$ is symmetric with respect to the weight
$W$. It follows from Theorem \ref{thm:equations_symmetry}, with $F_2=0$, that $E$ is symmetric
if and only if
\begin{align}
W(x)A_1(x)^*+A_1(x)W(x)&=0,\label{eq:conditions.symmetry2.OpE}\\
-(A_1(x)W(x))'+A_0W(x)&=W(x)A_0^*,\label{eq:conditions.symmetry3.OpE}
\end{align}
with the boundary condition
\begin{equation}
\label{eq:boundary.conditions.OpE}
\lim_{x\to \pm 1} A_1(x)W(x)=0.
\end{equation}
The second statement will then follow from Propositions \ref{prop:eigenvalue} and \ref{prop:Dsymmetric_D(W)}.

The verification of \eqref{eq:conditions.symmetry2.OpE} and \eqref{eq:conditions.symmetry3.OpE} involves elaborate computations, see Appendix \ref{appendix_b}.
\end{proof}

\begin{theorem} \label{thm:first_order_do}
Let $D$ be the second order matrix valued differential operator
\begin{equation*}
D=(1-x^2)\frac{d^2}{dx^2}+\left(\frac{d}{dx}\right)B_1(x)+B_0,
\end{equation*}
where the matrices $B_1(x)$ and $B_0$ are given by
\begin{align*}
B_1(x)&=\sum_{i=0}^{2\ell} \frac{(4\ell+3)(i-2\ell)}{2\ell} E_{i,i+1} 
-\sum_{i=0}^{2\ell} x\frac{(2\ell+3)(i-2\ell)}{\ell}E_{ii} + \sum_{i=0}^{2\ell} \left(\frac{3i}{2\ell}\right) E_{i,i-1},\\
B_0  &=\sum_{i=0}^{2\ell} \frac{(i-2\ell)(i\ell-2\ell^2-5\ell-3)}{2\ell} E_{ii}.
\end{align*}
Then $D$ is symmetric with respect to the weight $W(x)$; hence $D\in D(W)$. Moreover for every integer $n\geq 0$,
$$P_n(x)D=\Lambda_n(D)P_n(x),$$
where
$$\Lambda_n(D)= \sum_{i=0}^{2\ell} \left(n(n-1)-\frac{(2\ell+3)(i-2\ell)}{\ell} 
+\frac{(i-2\ell)(i\ell-2\ell^2-5\ell-3)}{2\ell} \right)E_{ii}.$$
\end{theorem}
\begin{proof}
The proof of the theorem is similar to that of Theorem \ref{conj:first_order_do},
see Appendix \ref{appendix_b}.
\end{proof}

\begin{cor}
 The differential operators $D$ and $E$ commute.
\end{cor}
\begin{proof}
 To see that $D$ and $E$ commute it is enough to
verify that the corresponding eigenvalues commute. The eigenvalues
commute because they are diagonal matrices.
\end{proof}

As we pointed out in Theorem \ref{thm:W_splits}, for any $\ell\in \frac12 \mathbb{N}$ the matrix weight $W$ 
and the polynomials $P_n$ are $(2\ell+1)\times(2\ell+1)$ matrices that 
can be conjugated into $2\times 2$ block matrices. More precisely
\begin{align*}
YW(x)Y^t&=\begin{pmatrix} W_1(x) & 0 \\ 0 & W_2(x) \end{pmatrix},  \quad 
YP_n(x)Y^{-1}=\begin{pmatrix} P_{n,1}(x) & 0 \\ 0 & P_{n,2}(x)\end{pmatrix},
\end{align*}
where $Y$ is the orthogonal matrix introduced in \eqref{eq:definition_M} and $W_1$, $W_2$ are the
square matrices described in Corollary \ref{thm:W_splits_even_case}.
Here $\{P_{n,1}\}_{n\geq0}$ and $\{P_{n,2}\}_{n\geq0}$ are the sequences of monic orthogonal polynomials
with respect to the weights $W_1$ and $W_2$ respectively.
\begin{prop}\label{prop:conj_first_order_DO}
Suppose $\ell=(2n+1)/2$ for some integer $n$, then $E$ splits in
$(n+1)\times(n+1)$ blocks in
the following way
\begin{equation*}
Y \, E \, Y^{t}=\widetilde E= \begin{pmatrix} -(\ell+1)I_{n+1} & E_1 \\ E_2 & -(\ell+1)I_{n+1} \end{pmatrix},
\end{equation*}
where
\begin{align*}
E_1&=\left(\frac{d}{dx}\right)\widetilde A_1(x) + \widetilde A_0,\\
E_2&= \left(\frac{d}{dx}\right) F_{n+1}J_{n+1} \widetilde A_1(-x) J_{n+1} F_{n+1} 
+F_{n+1}J_{n+1} \widetilde A_0 J_{n+1} F_{n+1}.
\end{align*}
Here $F_{n+1}$, $J_{n+1}$ are the matrices introduced in \eqref{eq:matrices_Jn_Fn}. The matrices $A_1$ and $A_0$ are given by
\begin{align*}
\widetilde A_1(x)&=-\sum_{i=0}^{n-1} \frac{(2\ell-i)}{2\ell}E_{i,n-i-1}+
x\sum_{i=0}^{n} \frac{(\ell-i)}{\ell}E_{i,n-i}
+\sum_{i=1}^{n} \frac{i}{2\ell}E_{i,n-i+1}+\frac{(2\ell+1)}{4\ell}E_{n,n+1},\\
\widetilde A_0&=\sum_{i=0}^{n} \frac{(\ell+1)(\ell-i)}{\ell} E_{i,n-i}.
\end{align*}
\end{prop}
\begin{proof}
 The proposition follows by a straightforward computation.
\end{proof}

\begin{prop}\label{prop:conj_first_order_DO_odd}
Suppose $\ell\in\mathbb{N}$, then we have
\begin{equation*}
Y \, E \, Y^{t}=\widetilde E= \left(\frac{d}{dx}\right) 
\begin{pmatrix} O_{(\ell+1)} &  \begin{matrix} \widetilde A_1(x) \\ v_1^t \end{matrix} \\  
\begin{matrix} F_\ell \widetilde A_1(x) F_\ell & v_2 \end{matrix} & O_{\ell\times \ell} \end{pmatrix}
+\begin{pmatrix} -(\ell+1)I_{(\ell+1)} &  \begin{matrix} \widetilde A_0(x) \\ v_0^t \end{matrix}\\  
\begin{matrix} F_\ell \widetilde A_0(x) F_\ell & v_0 \end{matrix} & -(\ell+1)I_\ell \end{pmatrix},
\end{equation*}
where $\widetilde A_1$ and $\widetilde A_0$ are $n\times n$ matrices given by
\begin{align*}
\widetilde A_1(x)&=-\sum_{i=0}^{\ell-2} \frac{(2\ell-i)}{2\ell}E_{i,\ell-i-1}+
x\sum_{i=0}^{\ell-1} \frac{(\ell-i)}{\ell}E_{i,\ell-i}+\sum_{i=1}^{\ell-1} \frac{i}{2\ell}E_{i,\ell-i+1},\\
\widetilde A_0&=\sum_{i=0}^{\ell-1} \frac{(\ell+1)(\ell-i)}{\ell} E_{i,\ell-i},
\end{align*}
and the vectors $v_0, v_1, v_2\in \mathbb{C}^\ell$ are
\begin{align*}
v_0=\left(0,0,\cdots,0\right),\quad v_1=\left(\frac{(2\ell+1)\sqrt{2}}{4\ell},0,\cdots,0\right),\quad 
v_2=\left(-\frac{(2\ell+1)\sqrt{2}}{4(\ell+1)},0,\cdots,0\right).
\end{align*}
\end{prop}
\begin{proof}
 The proposition follows by a straightforward computation.
\end{proof}

Let us assume that $\ell=(2n+1)/2$ for some $n\in \mathbb{N}$ so that 
the weight $W$ and the polynomials $P_n$ are matrices of even dimension. 
Proposition \ref{prop:conj_first_order_DO} says that
\begin{equation}\label{eq:equation_for_operator_E_conjugated}
\widetilde P_n(x) \widetilde E=Y\Lambda_n Y^t \widetilde P_n(x),\quad n\geq 0.
\end{equation}
A simple computation shows that
$$Y\Lambda_n Y^t=-(\ell+1)I_{2n+2}+\begin{pmatrix} 0 & \Lambda_{n,1} \\  \Lambda_{n,2} & 0 \end{pmatrix},$$
where $\Lambda_{n,1}$ is a $(n+1)\times(n+1)$ matrix (depending on $n$) and 
$$\Lambda_{n,2}=F_{n+1}J_{n+1} \Lambda_{n,1} J_{n+1} F_{n+1}.$$
It follows from \eqref{eq:equation_for_operator_E_conjugated} that the following matrix equation is satisfied
$$\begin{pmatrix} -(\ell+1)P_{n,1}(x) & P_{n,1}(x)E_1 \\ P_{n,2}(x)E_2 & -(\ell+1)P_{n,2}(x) \end{pmatrix}=
\begin{pmatrix} -(\ell+1)P_{n,1}(x) & \Lambda_{n,1}P_{n,2}(x) 
\\ \Lambda_{n,2} P_{n,1}(x) & -(\ell+1)P_{n,2}(x) \end{pmatrix}.$$
Therefore the polynomials $P_{n,1}$ and $P_{n,2}$ satisfy the following differential equations
\begin{align}\label{eq:diff_equations_P1n_P2n}
P_{n,1}E_1-\Lambda_{n,1}P_{n,2}&=0, \\ \label{eq:diff_equations_P1n_P2n_2}
P_{n,2}E_2-\Lambda_{n,2}P_{n,1}&=0.
\end{align}
Finally it follows from \eqref{eq:diff_equations_P1n_P2n}, \eqref{eq:diff_equations_P1n_P2n_2} 
and \eqref{eq:equation_Pd2} that for every $n\geq 0$, the 
polynomial $P_{n,1}$ is a solution of the following second-order matrix valued differential
equation
$$P_{n,1}E_1E_2-\Lambda_{n,1}\Lambda_{n,2}P_{n,1}=0.$$
We can also obtain a second order differential equation for $P_{n,2}$.


\section{Examples}\label{8}
The purpose of this section is to study the
properties of the monic orthogonal polynomials $\{P_n\}_{n\geq 0}$
presented in Section \ref{sec:orthogonal_polynomials} for small dimension.
For $\ell=0,\frac12,1,\frac23, 2$, we 
show that these polynomials are solutions of certain matrix valued differential equations.
We will show that the polynomials can be defined by means of Rodrigues' formulas and we
will give explicit expressions for the three term recurrence relations.

\subsection{The case $\ell=0$; the scalar weight}
In this case the polynomials $\{P_n\}_{n\geq 0}$ are scalar valued. The weight $W$
reduces to the real function
\begin{equation*}
W(x)=(1-x)^{\frac12}(1+x)^{\frac12}, \quad x\in[-1,1].
\end{equation*}
Therefore the polynomials $P_n$ are a multiple of the Chebyshev polynomials of the second
kind: $P_n(x)=2^{-n}U_n(x)$, $n\in\mathbb{N}$.

\subsection{The case $\ell=\frac12$; weight of dimension $2$}
In this case the polynomials $\{P_n\}_{n\geq 0}$ are $2\times 2$ matrices. The weight $W$
is given by
\begin{equation*}
W(x)=(1-x)^{\frac12}(1+x)^{\frac12} \begin{pmatrix} 2 & 2x \\ 2x & 2 \end{pmatrix}, \quad x\in[-1,1].
\end{equation*}
It is a straightforward computation that
\begin{equation*}
YW(x)Y^t=2\begin{pmatrix}  (1-x)^{\frac12}(1+x)^{\frac32} & 0 \\ 0 & (1-x)^{\frac32}(1+x)^{\frac12}
\end{pmatrix},
\quad Y=\frac{1}{\sqrt{2}}\begin{pmatrix} 1 & 1 \\ -1 & 1 \end{pmatrix}.
\end{equation*}
Observe that $W_1(x)=(1-x)^{\frac12}(1+x)^{\frac32}$ and $W_2(x)=(1-x)^{\frac32}(1+x)^{\frac12}$ are
Jacobi weights and therefore we have
\begin{equation*}
P_{n,1}=\frac{2^nn!(n+2)!}{(2n+2)!}P_n^{(\frac12,\frac32)}(x), 
\quad P_{n,2}(x)=\frac{2^nn!(n+2)!}{(2n+2)!}P_n^{(\frac32,\frac12)}, \quad n\in\mathbb{N}_0,
\end{equation*}
where $\{P_n^{(\alpha,\beta)}\}_{n\geq0}$ are the classical Jacobi polynomials

\subsubsection{Differential equations} By Theorem \ref{conj:first_order_do} we have
$$\frac{d}{dx}P_n(x)\begin{pmatrix} -x & 1 \\ -1 & x \end{pmatrix}
+ P_n(x)\begin{pmatrix} -3 & 0 \\ 0 & 0 \end{pmatrix} =
\begin{pmatrix} -n-3 & 0 \\ 0 & n \end{pmatrix} P_n(x).$$
We can conjugate the differential operator $E$ by the matrix $Y$ to obtain
\begin{equation*}
\widetilde E= Y \,E \, Y^{t}= \left(\frac{d}{dx}\right)\begin{pmatrix} 0 & 1+x \\ x-1 & 0 \end{pmatrix}
+\frac32\begin{pmatrix} -1 & 1 \\ 1 & -1 \end{pmatrix},
\end{equation*}
The monic polynomials
\begin{equation*}
\widetilde P_n(x)=YP_n(x)Y^{t}=\begin{pmatrix} P_{n,1}(x) & 0 \\ 0 & P_{n,2}(x)
 \end{pmatrix}, \quad n\in\mathbb{N}_0,
\end{equation*}
satisfy
\begin{equation*}
\widetilde P_n(x)\widetilde E =\widetilde \Lambda_n \widetilde P_n(x),
\text{ where }\widetilde \Lambda_n(x)=\begin{pmatrix} -3/2 & n+3/2 \\ n+3/2 & -3/2 \end{pmatrix}.
\end{equation*}
Now the fact that $\widetilde P_n(x)$ is an eigenfunction of $\widetilde E$
is equivalent to the following relations between Jacobi polynomials
\begin{align*}
(1+x)\frac{d}{dx}P^{(\frac12,\frac32)}(x)+\frac32 P^{(\frac12,\frac32)}_n(x)
-(n+\frac32)P^{(\frac32,\frac12)}_n(x)=0,\\
(1-x)\frac{d}{dx}P_n^{(\frac32,\frac12)}(x) +\frac32 P^{(\frac32,\frac12)}_n(x)
-(n+\frac32)P^{(\frac12,\frac32)}_n(x)=0.
\end{align*}

\subsection{Case $\ell=1$; weight of dimension $3$}
Here we consider the simplest example of nontrivial matrix orthogonal
polynomials for the weight $W$. The weight matrix $W$ of size $3\times
3$ is obtained by setting $\ell=1$. We have
\begin{equation}W(x)=(1-x)^{\tfrac12}(1+x)^{\tfrac12}
\begin{pmatrix}
3 & 3x & 4x^2-1 \\
3x & x^2+2 & 3x \\
4x^2-1 & 3x & 3
\end{pmatrix}
\end{equation}
We know from Theorem \ref{thm:W_splits} that the weight $W(x)$ splits into
a block of size $2\times 2$ and a block of size $1\times 1$, namely
$$YW(x)Y^{t}=\begin{pmatrix} W_1(x) & 0 \\ 0 & W_2(x)
\end{pmatrix}=(1-x)^{\frac12}(1+x)^{\frac12}\begin{pmatrix} 4x^2+2 & 3\sqrt{2}x & 0 \\ 3\sqrt{2}x & x^2+2 & 0
\\  0 & 0 & 4(1-x^2) \end{pmatrix}. $$
From Theorem \ref{thm:W_splits} the monic orthogonal polynomials
$\tilde P_{n}(x)$ with respect to $\tilde W(x)$  reduce to
$$\tilde P_n = \begin{pmatrix}
P_{n,1}(x) & 0 \\
0 & P_{n,2}(x) \end{pmatrix}$$
where $\{P_{n,2}\}_{n\geq 0}$ are the monic polynomials with respect to
$W_1(x)$ and $\{P_{n,2}\}_{n\geq 0}$ are the monic polynomials with
respect to the weight $W_2(x)$.
\begin{remark}
The weight $W_2$ is a multiple of the Jacobi weight
$(1-x)^\alpha(1+x)^\beta$ corresponding to $\alpha=3/2$ and $\beta=3/2$.
The monic polynomials $\{P_{n,2}\}_{n\geq 0}$  are then a multiple of the Gegenbauer polynomials
$$P_{n,2}(x)=\frac{2^nn!(n+3)!}{(2n+3)!}P^{(\frac32,\frac32)}_n(x).$$
\end{remark}
\subsubsection{The first order differential operator}
By Theorem \ref{conj:first_order_do} we have that the monic
polynomials $P_n$ are eigenfunctions of the differential operator $E$. 
More precisely the following equation holds
$$\frac{d}{dx}P_n(x)\begin{pmatrix} -x & 1 & 0 \\
-\frac12 & 0 & \frac12 \\ 0 & -1 &  x \end{pmatrix} + P_n(x)
\begin{pmatrix} -4 & 0 & 0 \\ 0& -2 & 0  \\ 0 & 0 & 0 \end{pmatrix} =
\begin{pmatrix} -n-4 & 0 & 0 \\ 0 & -2 & 0 \\0 & 0 & n \end{pmatrix} P_n(x).$$
Now we can conjugate the differential operator $E$ by the matrix $Y$ to obtain a differential
operator $\widetilde E = Y\,E \, Y^{t}$. The fact that the polynomials $P_n$ are eigenfunctions of $E$
says that the polynomials $\widetilde P_n$ are eigenfunctions of $\widetilde E$. In other words
\begin{gather*}
\frac{d}{dx}\widetilde P_n(x) \begin{pmatrix} 0 & 0 & x \\ 0 & 0 & \frac{\sqrt{2}}{2}
\\ x & -\sqrt{2} &  0 \end{pmatrix}
+\widetilde P_n(x)\begin{pmatrix} -2 & 0 & 2 \\ 0& -2 & 0  \\ 2 & 0 & -2 \end{pmatrix} 
=\begin{pmatrix} -2 & 0 & n+2 \\ 0& -2 & 0  \\ n+2 & 0 & -2 \end{pmatrix} \widetilde P_n(x).
\end{gather*}
We can now rewrite the equation above in terms of the polynomials $P_{n,1}$ and
$P_{n,2}$.
\begin{align*}
\frac{d}{dx}  P_{n,1}(x)\begin{pmatrix} x \\ \frac{\sqrt2}{2} \end{pmatrix}
+   P_{n,1}(x) \begin{pmatrix} 2 \\ 0 \end{pmatrix} &=
 P_{n,2}(x) \begin{pmatrix} n+2 \\ 0 \end{pmatrix},\\
\frac{d}{dx}  P_{n,2}(x)\begin{pmatrix} x & -\sqrt2 \end{pmatrix} +
 P_{n,2}(x) \begin{pmatrix} 2 & 0 \end{pmatrix}
&= \begin{pmatrix} n+2 & 0 \end{pmatrix}P_{n,1}(x).
\end{align*}
Since $P_{n,2}$ is a Gegenbauer polynomial, we see that the elements of
the first row of $P_{n,2}$ can be written explicitly in terms of 
Gegenbauer polynomials.

\subsubsection{Second order differential operators}
In this subsection we describe a set of linearly independent
differential operators that have the polynomials $P_{n,1}$ as eigenfunctions.
\begin{prop}
The matrix orthogonal polynomials $\{P_{n,1}\}_{n\geq 0}$ satisfy
\begin{equation*}
P_{n,1} D_j=\Lambda_n(D_j)P_{n,1}, \quad j=1,2,3, n\geq 0,
\end{equation*}
where the differential operators $D_j$ are
\begin{align*}
D_1&=(x^2-1)\left(\frac{d^2}{d x^2}\right)+\left(\frac{d}{d x}\right)\begin{pmatrix} 5x & -4 \\ -1 & 5x
\end{pmatrix} + \begin{pmatrix} 1 & 0 \\ 0 & 0 \end{pmatrix},\\
D_2&=\left(\frac{d^2}{d x^2}\right) \begin{pmatrix} x^2 & -2x \\ \tfrac{x}{2} & -1
\end{pmatrix} +\left(\frac{d}{d x}\right) \begin{pmatrix} 5x & -6 \\ 1 & 0 \end{pmatrix} +
 \begin{pmatrix} 4 & 0 \\ 0 & 0 \end{pmatrix} ,\\
D_3&=\left(\frac{d^2}{d x^2}\right)\begin{pmatrix} -2x & 8x^2-4 \\ x^2-2 & 2x
\end{pmatrix} + \left(\frac{d}{d x}\right) \begin{pmatrix} -8 & 32x \\ 6x & -4 \end{pmatrix}
+\begin{pmatrix} 0 & 16 \\ 6 & 0 \end{pmatrix}.
\end{align*}
and the eigenvalues $\Lambda_j$ are given by
\begin{align*}
\Lambda_n(D_1)&=\begin{pmatrix} n(n+4) + 1 & 0 \\ 0 & n(n+4)
\end{pmatrix}, \quad \Lambda_n(D_2)=\begin{pmatrix} (n+2)^2 & 0 \\ 0 & 0
\end{pmatrix},\\
\Lambda_n(D_3)&=\begin{pmatrix} 0 & 8(n+2)(n+1) \\ (n+3)(n+2) & 0
\end{pmatrix}.
\end{align*}
Moreover, the differential operators $D_1$, $D_2$ and $D_3$ satisfy
\begin{align*}
D_1D_2&=D_2D_1, \quad D_1D_3\neq D_3D_1, \quad D_2D_3\neq D_3D_2.
\end{align*}
\end{prop}
\begin{proof}
The proposition follows by proving that the differential operators $D_j$, $j=1,2,3$ are
symmetric with respect to the weight matrix $W_1$. This is accomplished by a straightforward computation,
showing that the differential equations \eqref{eq:conditions.symmetry1}, \eqref{eq:conditions.symmetry2},
\eqref{eq:conditions.symmetry3} and the boundary conditions \eqref{eq:boundary.conditions} are satisfied.
As a consequence of Remark \ref{rmk:eigenvalues_are_representation}, the commutativity properties of the
differential operators follow by observing the commutativity of the corresponding eigenvalues.
\end{proof}

\subsubsection{Rodrigues' Formula}

\begin{prop}
The matrix orthogonal polynomials $\{P_{n,1}(x)\}_{n\geq 0}$ satisfy the
Rodrigues' formula
\begin{equation}\label{eq:rodrigues_formula_P1n}
P_{n,1}(x)=c\left[(1-x^2)^{\tfrac12+n}\left(\begin{pmatrix} 4x^2+2 & 3\sqrt{2}x \\ 3\sqrt{2}x & x^2+2
\end{pmatrix} +\begin{pmatrix}
\tfrac{2n}{n+2} & \tfrac{\sqrt{2}nx}{n+2} \\ -\tfrac{\sqrt{2}nx}{n+1} & -\tfrac{n}{n+1}
\end{pmatrix}\right) \right]^{(n)} \hskip-3pt W_1^{-1}(x),
\end{equation}
where
$$c=\frac{(-1)^n2^{-2n-2}(n+2)(n+3)\sqrt{\pi}}{(2n+3)\Gamma(n+\frac32)}.$$
\end{prop}
\begin{proof}
The proposition can be proven in a similar way to Theorem 3.1 of \cite{DG3}. We include a sketch
of the proof for the for the sake of completeness. First of all, we recall that the classical
Jacobi polynomials $P_n^{(\alpha,\beta)}(x)$ satisfy the Rodrigues' formula
$$P_n^{(\alpha,\beta)}(x)=\frac{(-1)^n}{2^nn!}(1-x)^{-\alpha}(1+x)^{-\beta}
[(1-x)^{\alpha+n}(1+x)^{\beta+n}]^{(n)}.$$
Let $R(x)$ and $Y_n$ be the matrix polynomials of degree $2$ and $1$
$$R(x)=\begin{pmatrix} 4x^2+2 & 3\sqrt{2}x \\ 3\sqrt{2}x & x^2+2 \end{pmatrix},
\qquad Y_n(x) =\begin{pmatrix} \tfrac{2n}{n+2} & \tfrac{n\sqrt{2}x}{n+2}
\\ -\tfrac{\sqrt{2}nx}{n+1} & -\tfrac{n}{n+1}
\end{pmatrix},$$
so that the \eqref{eq:rodrigues_formula_P1n} can be rewritten as
$$P_{n,1}(x)=c\left[(1-x^2)^{\tfrac12+n}\left(R(x) + Y_n(x) \right) \right]^{(n)}W_1^{-1}(x).$$
Then by applying the Leibniz rule on the right hand side of \eqref{eq:rodrigues_formula_P1n},
it is not difficult to prove that
\begin{multline*}
P_{n,1}(x)=c\tfrac12n(n-1)[(1-x)^{\frac12+n}(1+x)^{\frac12+n}]^{(n-2)}(1-x)^{-\frac12}(1+x)^{-\frac12}
R''(x)R(x)^{-1}\\
+cn[(1-x)^{\frac12+n}(1+x)^{\frac12+n}]^{(n-1)}(1-x)^{-\frac12}(1+x)^{-\frac12}(R'(x)+Y'(x))R(x)^{-1}\\
+c[(1-x)^{\frac12+n}(1+x)^{\frac12+n}]^{(n)}(1-x)^{-\frac12}(1+x)^{-\frac12}(I+Y(x)R(x)^{-1}).
\end{multline*}
Now by applying the Rodrigues' formula for the Jacobi polynomials we obtain
\begin{multline}\label{eq:P1n_in_terms_of_Jacobi}
P_{n,1}(x)=2^nn!(-1)^nc[ P_n^{(\frac12,\frac12)}(x)(I+Y(x)R(x)^{-1}) \\
-\tfrac12 (1-x)(1+x)P_{n-1}^{(\frac32,\frac32)}(x)(Y'(x)+R'(x))R(x)^{-1}\\
+\tfrac18 (1-x)^2(1+x)^2 P_{n-2}^{(\frac52,\frac52)}(x)R''(x)R(x)^{-1}].
\end{multline}
Now with a careful computation we can show that the expression above is a matrix polynomial of degree $n$
with nonsingular leading term.
Using integration by parts it is easy to show the orthogonality of $P_{n,1}$ and $x^m$,
$m=0,1,\ldots,n-1$, with respect to the weight $W_1$.
\end{proof}

\subsubsection{Three term recurrence relations}
In Corollary \ref{thm:three_term_for_Pd} we show that the matrix polynomials
$P_n(x)$ of any size satisfy a three term recurrence relation. The recurrence relation for the
polynomials $P_{n,1}$ can then be obtained by conjugating the recurrence relation for
$P_n(x)$ by the matrix $Y$. The recurrence coefficients \eqref{eq:coefficients_three_term}
are given in terms  of Clebsch-Gordan coefficients and are difficult to manipulate.
For $\ell=1$ we can use the Rodrigues' formula \eqref{eq:rodrigues_formula_P1n} to
derive explicit formulas for the three term recurrence relation for the polynomials $P_{n,1}$.

First we need to compute the norm of $P_{n,1}(x)$. The Rodrigues' formula \eqref{eq:rodrigues_formula_P1n}
and integration by parts lead to
$$\| P_{n,1} \|^2=2^{-2n-1}\pi \begin{pmatrix} \frac{(n+3)}{(n+1)} & 0 \\ 0 &
\frac{(n+3)^2}{8(n+1)^2}\end{pmatrix}.$$
If $\{\mathcal{P}_{n,1}\}_{n\geq 0}$ is a sequence of orthonormal polynomials with respect to $W_1$
with leading coefficients $\Omega_n$, then it follows directly from the orthogonality relations
for the monic polynomials that
$\|P_{n,1} \| ^2 = \Omega_n^{-1}(\Omega_n^*)^{-1}$.
The orthonormal polynomials $\mathcal{P}_{n,1}$ with leading coefficient
$$\Omega_n=\begin{pmatrix} \sqrt{\frac{2^{2n+1}(n+1)}{\pi(n+3)}} & 0 
\\ 0 & \frac{2^{n+2}(n+1)}{\sqrt{\pi}(n+3)} \end{pmatrix},$$
satisfy the three term recurrence relation
$$x\mathcal{P}_n(x)=A_{n+1}\mathcal{P}_{n+1}(x)+B_n\mathcal{P}_n(x)+A_{n}^*\mathcal{P}_{n-1}(x),$$
where $A_{n}=\Omega_{n-1}\Omega_{n}^{-1}$ and
$$B_n=\Omega_n[\text{coef. of }x^{n-1}\text{ in }P_{n,1} - \text{coef. of }x^{n}\text{ in }P_{n+1,1}]\Omega_n^{-1}.$$
The coefficient of $x^{n-1}$ in $P_{n,1}$ can by obtained from \eqref{eq:P1n_in_terms_of_Jacobi}.
Now a careful computation shows that
$$A_n=\begin{pmatrix} \frac12\sqrt{\frac{n(n+3)}{(n+1)(n+2)}} &
0 \\ 0 & \frac{n(n+3)}{2(n+1)(n+2)}\end{pmatrix},\quad
B_n=\begin{pmatrix} 0 & \frac{4}{(n+2)(n+3)} \\
\frac{1}{2(n+1)(n+4)} & 0\end{pmatrix}.$$
Therefore the monic polynomials $P_{n,1}$ satisfy the three term
recurrence relation
$$xP_{n,1}(x)=P_{n+1,1}(x)+\tilde B_n P_{n,1}(x) + \tilde C_n P_{n-1,1}(x),$$
where
$$\tilde B_{n}=\begin{pmatrix} 0 & \tfrac{4}{(n+2)(n+3)} \\
\tfrac{1}{2(n+1)(n+2)} & 0 \end{pmatrix},
\qquad \tilde C_{n} = \begin{pmatrix} \frac{n(n+3)}{4(n+2)(n+1)} & 0 \\ 0
& \frac{n^2(n+3)^2}{4(n+2)^2(n+1)^2} \end{pmatrix}.$$

\subsection{Case $\ell=3/2$; weight of dimension 4}
The weight matrix $W$ of size $4\times 4$ is obtained by setting $\ell=3/2$.
\begin{equation*}
W(x)=(1-x^2)^{\tfrac12}\begin{pmatrix}
4 & 4x & \tfrac43 (4x^2-1) & 4x(2x^2-1) \\
4x & \tfrac49(4x^2+5) & \tfrac49x(2x^2+7) & \tfrac43(4x^2-1) \\
\tfrac43(4x^2-1) & \tfrac49x(2x^2+7) & \tfrac49(4x^2+5) & 4x \\
4x(2x^2-1) & \tfrac43(4x^2-1) &4x &4
\end{pmatrix}
\end{equation*}
We know from Corollary \ref{thm:W_splits} that the weight $W(x)$ splits in
two blocks of size $2\times 2$, namely
$$\tilde W(x) = YW(x)Y^{t}=\begin{pmatrix} W_1(x) & 0 \\ 0 & W_2(x)
\end{pmatrix}, $$
where
$$W_1(x)=4(1-x)^{1/2}(1+x)^{3/2}\begin{pmatrix}
2x^2-2x+1 & \tfrac13(4x-1) \\ \tfrac13(4x-1)  & \tfrac19 (2x^2+2x+5)
\end{pmatrix},$$
and
$$W_2(x)=J_2F_2 W_1(-x) F_2J_2, \quad \text{where } F_2=\begin{pmatrix} 1 & 0 \\ 0 & -1
\end{pmatrix} \text{ and } J_2=\begin{pmatrix} 0 & 1 \\ 1 & 0 \end{pmatrix}.$$
It follows from Corollary \ref{thm:W_splits_even_case} 
that the monic orthogonal polynomials  $P_{n,2}$ with
respect to the weight $W_2$ are completely determined by 
the the monic orthogonal polynomials $P_{n,1}$ with 
respect to $P_{n,2}(x)=J_2F_2P_{n,1}(-x)F_2J_2.$
Therefore we only need to study the polynomials $P_{n,1}$.

\subsubsection{Differential operators}
In this subsection we describe a set of linearly independent
differential operators that have the polynomials $P_{n,1}$ as eigenfunctions.
\begin{prop}
The matrix orthogonal polynomials $\{P_{n,1}\}_{n\geq 0}$ satisfy
\begin{equation*}
P_{n,1} D_j=\Lambda_n(D_j)P_{n,1}, \quad j=1,2,3, n\geq 0,
\end{equation*}
where the differential operators $D_j$ are
\begin{align*}
D_1&=(x^2-1)\left(\frac{d^2}{d x^2}\right)+\left(\frac{d}{d x}\right)\begin{pmatrix} 6x & -3 \\ -1 & 6x-2
\end{pmatrix} + \begin{pmatrix} 2 & 0 \\ 0 & 0 \end{pmatrix},\\
D_2&=\left(\frac{d^2}{d x^2}\right)\begin{pmatrix} x^2 - \tfrac14 & -\tfrac32x+\tfrac34
\\ \tfrac{x}{2} +\tfrac14 & -\tfrac34 \end{pmatrix} + \left(\frac{d}{d x}\right)
\begin{pmatrix} 6x & \tfrac92 \\ \tfrac32 & 0 \end{pmatrix} + \begin{pmatrix} 6 & 0 \\ 0 & 0 \end{pmatrix} ,\\
D_3&=\left(\frac{d^2}{d x^2}\right)\begin{pmatrix} -3x+3 & 9x^2-9x \\ x^2+x-2 & 3x-3 \end{pmatrix} \\
& \hspace{4cm}+ \left(\frac{d}{d x}\right)
\begin{pmatrix} -9 & 36x-18 \\ 8x+4 & -3 \end{pmatrix} + \begin{pmatrix} 0 & 18 \\ 12 & 0 \end{pmatrix},
\end{align*}
and the eigenvalues $\Lambda_j$ are given by
\begin{align*}
\Lambda_n(D_1)&=\begin{pmatrix} n(n+5) + 2 & 0 \\ 0 & n(n+5)
\end{pmatrix}, \quad \Lambda_n(D_2)=\begin{pmatrix} (n+3)(n+2) & 0 \\ 0 &
0 \end{pmatrix},\\
\Lambda_n(D_3)&=\begin{pmatrix} 0 & 9(n+2)(n+1) \\ (n+4)(n+3) & 0 \end{pmatrix}.
\end{align*}
Moreover, the differential operators $D_1$, $D_2$ and $D_3$ satisfy
\begin{align*}
D_1D_2&=D_2D_1, \quad D_1D_3\neq D_3D_1, \quad D_2D_3\neq D_3D_2.
\end{align*}
\end{prop}

\subsubsection{Rodrigues' Formula}
The monic orthogonal polynomials $\{P_{n,1}(x)\}_{n\geq 0}$ satisfy the
Rodrigues' formula
$$P_{n,1}(x)=c\left[(1-x)^{\tfrac12+n}(1+x)^{\tfrac32+n}(R(x)+Y_n(x))
\right]^{(n)}W_1^{-1}(x),$$
where
$$c=\frac{2^{-2n-2}(-1)^n(n+3)(n+4)\sqrt{\pi}}{\Gamma(n+\frac52)},$$
and
\begin{align*}
R(x)&=\begin{pmatrix} 2x^2-2x+1 & \tfrac13(4x-1) \\ \tfrac13(4x-1)  &
\tfrac19 (2x^2+2x+5) \end{pmatrix},\\
Y_n(x)&=\begin{pmatrix} \tfrac{n}{n+3} & \tfrac{n}{3(n+3)}(2x+1) \\
\tfrac{n}{3(n+1)}(1-2x) & -\tfrac{n}{3(n+1)} \end{pmatrix}.
\end{align*}
\subsubsection{Three term recurrence relations}
The orthonormal polynomials $\mathcal{P}_{n,1}(x)=\|P_{n,1}\|^{-1}P_{n,1}$,
with leading coefficient
$$\Omega_n=\begin{pmatrix} \sqrt{\frac{2^{2n+1}(n+1)}{\pi(n+4)}} & 0 
\\ 0 & \frac{9(n+1)\sqrt{2^{2n+1}(n+2)}}{\sqrt{\pi(n+3)}(n+4)} \end{pmatrix},$$
satisfy the three term recurrence relation
$$x\mathcal{P}_n(x)=A_{n+1}\mathcal{P}_{n+1}(x)+B_n\mathcal{P}_n(x)+A_{n}^*\mathcal{P}_{n-1}(x),$$
where
\begin{align*}
A_n&=\begin{pmatrix} \frac12\sqrt{\frac{n(n+4)}{(n+1)(n+3)}} &
0 \\ 0 & \frac{n(n+4)}{2(n+2)\sqrt{(n+1)(n+3)}}\end{pmatrix},\\
B_n&=\begin{pmatrix} 0 & \frac{3}{2\sqrt{(n+1)(n+2)(n+3)(n+4)}} \\
\frac{3}{2\sqrt{(n+1)(n+2)(n+3)(n+4)}} &  \frac{2}{(n+2)(n+3)}\end{pmatrix}.
\end{align*}
Therefore the monic polynomials $P_{n,1}(x)$ satisfy the three term
recurrence relation
$$xP_{n,1}=P_{n+1,1}+\tilde B_n P_{n,1} + \tilde C_n P_{n-1,1},$$
where
$$\tilde B_{n}=\begin{pmatrix} 0 & \tfrac{9}{2(n+3)(n+4)} \\
\tfrac{1}{2(n+1)(n+2)} & \tfrac{2}{(n+2)(n+3)} \end{pmatrix},
\qquad \tilde C_{n} = \begin{pmatrix} \frac{n(n+4)}{4(n+1)(n+3)} & 0 \\ 0
& \frac{n^2(n+4)^2}{4(n+1)(n+2)^2(n+3)} \end{pmatrix}.$$

\subsection{Case $\ell=2$; weight of dimension 5}
In this subsection we consider the $2\times 2$ 
irreducible block in the case $\ell=2$, where the matrix
weight $W$ is of dimension $5$. This case completes the list of all
irreducible $2\times 2$ blocks obtained by conjugating the weight $W$ by the matrix $Y$.

The $2\times 2$ block is given by
\begin{equation*}
 W_1(x)=(1-x)^{\frac32}(1+x)^{\frac32}\begin{pmatrix} x^2+4 & 10x \\ 10x & 16x^2+4 \end{pmatrix}.
\end{equation*}
As before, we denote by $\{P_{n,1}\}_n$ the sequence of monic orthogonal polynomials with respect to $W_1$.
\subsubsection{Differential operators}
In this subsection we describe a set of linearly independent
differential operators that have the polynomials $P_{n,1}$ as eigenfunctions.
\begin{prop}
The matrix orthogonal polynomials $\{P_{n,1}\}_{n\geq 0}$ satisfy
\begin{equation*}
P_{n,1} D_j=\Lambda_n(D_j)P_{n,1}, \quad j=1,2,3, n\geq 0,
\end{equation*}
where the differential operators $D_j$ are
\begin{align*}
D_1&=(x^2-1)\left(\frac{d^2}{d x^2}\right)+\left(\frac{d}{d x}\right)\begin{pmatrix} 7x & -1 \\ -4 & 7x
\end{pmatrix} + \begin{pmatrix} 3 & 0 \\ 0 & 0 \end{pmatrix},\\
D_2&=\left(\frac{d^2}{d x^2}\right)\begin{pmatrix} x^2 & -\tfrac12x \\ 2x & -1 \end{pmatrix} 
+ \left(\frac{d}{d x}\right)
\begin{pmatrix} 7x & -3 \\ 2 & 0 \end{pmatrix} + \begin{pmatrix} 5 & 0 \\ 0 & 0 \end{pmatrix} ,\\
D_3&=\left(\frac{d^2}{d x^2}\right)\begin{pmatrix} \tfrac38 x & \tfrac{x^2}{16}-\tfrac14 
\\ x^2-\tfrac14 & -\tfrac38 x \end{pmatrix} + \left(\frac{d}{d x}\right)
\begin{pmatrix} -\tfrac14 & \tfrac58 x \\ 4x & -1 \end{pmatrix} + \begin{pmatrix} 0 & \tfrac54 
\\ 2 & 0 \end{pmatrix},
\end{align*}
and the eigenvalues $\Lambda_j$ are given by
\begin{align*}
\Lambda_n(D_1)&=\begin{pmatrix} n(n+6)-3 & 0 \\ 0 & n(n+6)
\end{pmatrix}, \quad \Lambda_n(D_2)=\begin{pmatrix} (n+1)(n+5) & 0 \\ 0 &
0 \end{pmatrix},\\
\Lambda_n(D_3)&=\begin{pmatrix} 0 & \tfrac{1}{16}(n+5)(n+4) \\ (n+2)(n+1) & 0 \end{pmatrix}.
\end{align*}
Moreover, the differential operators $D_1$, $D_2$ and $D_3$ satisfy
\begin{align*}
D_1D_2&=D_2D_1, \quad D_1D_3\neq D_3D_1, \quad D_2D_3\neq D_3D_2.
\end{align*}
\end{prop}

\subsubsection{Rodrigues' Formula}
The monic orthogonal polynomials $\{P_{n,1}(x)\}_{n\geq 0}$ satisfy the
Rodrigues' formula
$$P_{n,1}(x)=c\left[(1-x)^{\tfrac32+n}(1+x)^{\tfrac32+n}(R(x)+Y_n(x))
\right]^{(n)}W_1^{-1}(x),$$
where
$$c=\frac{(-1)^n2^{-2n-4}(n+3)(n+4)(n+5)\sqrt{\pi}}{(2n+5)\Gamma(n+\frac52)},$$
and
\begin{align*}
R(x)=\begin{pmatrix} x^2+4 & 10x \\ 10x & 16x^2+4 \end{pmatrix}, \quad
Y_n(x)=\begin{pmatrix} -\tfrac{3n}{n+1} & -\tfrac{6nx}{n+1} \\
\tfrac{6nx}{n+4} & \tfrac{12n}{n+4} \end{pmatrix}.
\end{align*}
\subsubsection{Three term recurrence relations}
The orthonormal polynomials $\mathcal{P}_{n,1}(x)=\|P_{n,1}\|^{-1}P_{n,1}$ with leading coefficient
$$\Omega_n=\begin{pmatrix} 2^{2n+2}\frac{\sqrt{(n+2)}(n+1)}{\sqrt{\pi(n+4)}(n+5)} & 0 
\\ 0 & 2^n\frac{\sqrt{2(n+1)}}{\sqrt{\pi(n+5)}} \end{pmatrix},$$
satisfy the three term recurrence relation
$$x\mathcal{P}_n(x)=A_{n+1}\mathcal{P}_{n+1}(x)+B_n\mathcal{P}_n(x)+A_{n}^*\mathcal{P}_{n-1}(x),$$
where
\begin{align*}
A_n&=\begin{pmatrix} \frac{n(n+5)}{2\sqrt{(n+1)(n+2)(n+3)(n+4)}} &
0 \\ 0 & \frac{\sqrt{n(n+5)}}{\sqrt{(n+1)(n+4)}}\end{pmatrix},\\
B_n&=\begin{pmatrix} 0 & \frac{2}{\sqrt{(n+1)(n+2)(n+4)(n+5)}} \\
\frac{2}{\sqrt{(n+1)(n+2)(n+4)(n+5)}} &  \end{pmatrix}.
\end{align*}
Therefore the monic polynomials $P_{n,1}(x)$ satisfy the three term
recurrence relation
$$xP_{n,1}=P_{n+1,1}+\tilde B_n P_{n,1} + \tilde C_n P_{n-1,1},$$
where
$$\tilde B_{n}=\begin{pmatrix} 0 & \tfrac{1}{2(n+1)(n+2)} \\
\tfrac{8}{(n+4)(n+5)} & 0 \end{pmatrix},
\qquad \tilde C_{n} = \begin{pmatrix} \frac{n^2(n+5)^2}{4(n+1)(n+2)(n+3)(n+4)} & 0 \\ 0
& \frac{n(n+5)}{4(n+1)(n+4)} \end{pmatrix}.$$

\appendix

\section{Transformation formulas}\label{appendixA}
The goal of this appendix it to prove Theorem \ref{superweight}. We use the standard notation for the Pochhammer symbols and the hypergeometric series from \cite{AAR}. In the manipulations we only need the Chu-Vandermonde summation formula 
\cite{AAR}*{Cor. 2.2.3} which reads
\begin{equation}\label{formula:ChuVandermonde}
 \rFs{2}{1}{-n,a}{c}{1}=\frac{(c-a)_{n}}{(c)_{n}}.
\end{equation}
and Sheppard's transformation formula for $_{\,3}F_{2}$'s \cite{AAR}*{Cor. 3.3.4} written as 
\begin{equation}\label{sheppard's formula}
\begin{split}
 \sum_{k=0}^{n}(e+k)_{n-k}(d+k)_{n-k}\frac{(-n)_{k}(a)_{n}(b)_{n}}{k!}=\\
\sum_{k=0}^{n}(d-a)_{n-k}(e-a)_{n-k}\frac{(-n)_{k}(a)_{k}(a+b-n-d-e+1)_{k}}{k!}.
\end{split}
\end{equation}

\begin{prop}\label{prop:d_r non-constant}
Let $\ell\in\frac{1}{2}\bbN$ and $p,q\in\frac{1}{2}\bbZ$ such that $|p|,|q|\le\ell$, $\ell-p,\ell-q\in\bbZ$, $q-p\le0$ and $q+p\le0$. Let $s\in\{0,\ldots,\ell+q\}$ and define
\begin{equation}
e_{s}(p,q)=\sum_{n=0}^{\ell+q-s}\binom{\ell+p}{n}\binom{\ell+q}{n+s}\sum_{m=0}^{\ell-p-s}\frac{\binom{\ell-p}{m+s}\binom{\ell-q}{m}}{\binom{2 \ell}{m+n+s}^{2}}.
\end{equation}
Then we have
\begin{equation}
e_{s}^{\ell}(p,q)=\frac{(2\ell+1)}{(\ell+p+1)}\frac{(\ell-q)!(\ell+q)!}{(2\ell)!}\sum_{T=0}^{\ell+q-s}(-1)^{\ell+q-T}\frac{(p-\ell)_{\ell+q-T}(2+2\ell-T)_{T}}{(\ell+p+2)_{\ell+q-T}T!}.
\end{equation} 
\end{prop}
\begin{proof}
First we reverse the inner summation using $M=\ell-p-s-m$ to get
\begin{equation}\label{eqn:e_s, s neg}
e_{s}(p,q)=\sum_{n=0}^{\ell+q-s}\binom{\ell+p}{n}\binom{\ell+q}{n+s}\sum_{M=0}^{\ell-p-s}\frac{\binom{\ell-p}{M}\binom{\ell-q}{\ell-p-M-s}}{\binom{2 \ell}{\ell-p-M+n}^{2}}.
\end{equation}
We rewrite the inner summation:
\begin{multline}\label{eqn:innersum s pos}
 \sum_{M=0}^{\ell-p-s}\frac{\binom{\ell-p}{M}\binom{\ell-q}{\ell-p-M-s}}{\binom{2 \ell}{\ell-p-M+n}^{2}}=\\
\frac{(\ell-q)!}{(2\ell)!}(-1)^{s}(-\ell+p)_{s}\binom{2\ell}{\ell-p+n}^{-1}\sum_{M=0}^{\ell-p-s}\frac{(-\ell+p+s)_{M}(\ell+p-n+1)_{M}}{M!(-\ell+p-n)_{M}}B(M)
\end{multline}
where $B(M)=(\ell-p-M+1)_{n}(p-q+M+s+1)_{\ell+q-s-n}$ is a polynomial in $M$ of degree $\ell+q-s$ that depends on $\ell,p,q,n$ and $s$. The polynomial $B(M)$ has an expansion in $(-1)^{t}(-M)_{t}$,
\begin{equation}\label{eq: B s neg}
B(M)=\sum_{t=0}^{\ell+q}A_{t}\cdot(-1)^{t}(-M)_{t}.
\end{equation}
The coefficients $A_{t}=A_{t}(\ell,p,q,n,s)$ can be found by repeated application of the difference operator
$\Delta_{M}f=f(M+1)-f(M)$. Let $\Delta^{i}_{M}$ be its $i$-th power. 
We have
\begin{equation}\label{at}
\left.\frac{\Delta^{t}_{M}}{t!}\right|_{M=0}B=A_{t}.
\end{equation}
In other words,
\begin{equation}
B(M)=\sum_{t=0}^{\ell+q}A_{t}\cdot(-1)^{t}(-M)_{t}\quad\mbox{with $A_{t}=\left.\frac{\Delta^{t}_{M}}{t!}\right|_{M=0}B$}.
\end{equation}
We calculate (\ref{eqn:innersum s pos}) by substituting (\ref{eq: B s neg}) in it. Interchanging summations, the inner sum can be evaluated using (\ref{formula:ChuVandermonde}) (after shifting the summation parameter). We get
\begin{multline}\label{eqn:innersum s pos after Chu}
 \sum_{M=0}^{\ell-p-s}\frac{\binom{\ell-p}{M}\binom{\ell-q}{\ell-p-M-s}}{\binom{2 \ell}{\ell-p-M+n}^{2}}=\frac{(\ell-q)!}{(2\ell)!}(-1)^{s}(-\ell+p)_{s}\binom{2\ell}{\ell-p+n}^{-1}\times\\
\sum_{t=0}^{\ell+q-s}A_{t}\frac{(-2\ell-1)_{\ell-p-s-t}(-\ell+p+s)_{t}(\ell+p-n+1)_{t}}{(-\ell+p-n)_{\ell-p-s}}.
\end{multline}
Substituting (\ref{eqn:innersum s pos after Chu}) and (\ref{at}) in (\ref{eqn:e_s, s neg}) and simplifying gives
\begin{multline}
 e_{s}(p,q)=\frac{(\ell-q)!(\ell+p)!}{(2\ell)!(2\ell)!}(-1)^{\ell-p-s}(-\ell+p)_{s}(-\ell-q)_{s}
\sum_{t=0}^{\ell+q-s}(\ell+p+1)_{t}(-2\ell-1)_{\ell-p-s-t}(-\ell+p+s)_{t}\times\\
\left.\frac{\Delta^{t}_{M}}{t!}\right|_{M=0}(p-q+s+M+1)_{\ell+q-s}\sum_{n=0}^{\ell+q-s}\frac{(-\ell-q-s)_{n}(-\ell-p)_{n}(\ell-p-M+1)_{n}}{n!(-\ell-p-t)_{n}(-\ell-p-M)_{n}}.
\end{multline}
The inner sum over $n$ is a $_{\,3}F_{2}$-series, which can be transformed using (\ref{sheppard's formula}). Note that the $t$-order difference operator can now be evaluated yielding only one non-zero term in the sum over $n$. This gives
\begin{equation}
 e_{s}(p,q)=\frac{2\ell+1}{\ell+p+1}\frac{(\ell-q)!(\ell+q)!}{(2\ell)!} \sum_{t=0}^{\ell+q-s}(-1)^{-s-t}\frac{(-\ell+p)_{s+t}(2+\ell-q+s+t)_{\ell+q-s-t}}{(\ell+p+2)_{s+t}(\ell+q-s-t)!}.
\end{equation}
Reversing the order of summation using $T=\ell+q-s-t$ yields
\begin{equation}
e_{s}^{\ell}(p,q)=\frac{(2\ell+1)}{(\ell+p+1)}\frac{(\ell-q)!(\ell+q)!}{(2\ell)!}\sum_{T=0}^{\ell+q-s}(-1)^{\ell+q-T}\frac{(p-\ell)_{\ell+q-T}(2+2\ell-T)_{T}}{(\ell+p+2)_{\ell+q-T}T!}
\end{equation} 
as was to be shown.
\end{proof}

\begin{proof}[Proof of Theorem \ref{superweight}]
We already argued that there is an expansion in Chebyshev polynomials (\ref{eqn:c_s}). From (\ref{eq:weight4}) it follows that there are coefficients $d_{r}^{\ell}(p,q)$ such that
\begin{equation}\label{def: d_r}
 v^{\ell}_{p,q}(\cos t)=\sum_{r=-(\ell+\frac{p+q}{2})}^{\ell+\frac{p+q}{2}}d_{r}^{\ell}(p,q) 
{e^{-2irt}}.
\end{equation}
The coefficients $d_{r}^{\ell}(p,q)$ and $c_{n}^{\ell}(p,q)$ are related by
\begin{equation}
d_{r}^{\ell}(p,q)=\sum_{n=0}^{\ell+q-(r+\frac{q-p}{2})}c_{n}^{\ell}(p,q).
\end{equation}
Let $q-p\le0$, $q+p\le0$ and $r\ge\frac{p-q}{2}$ and substitute $r(s)=s+\frac{p-q}{2}$ in (\ref{eq:weight4}). Comparing this to (\ref{def: d_r}) shows 
\begin{equation} d_{r(s)}^{\ell}(p,q)=\sum_{n=0}^{\ell+q-s}\binom{\ell+p}{n}\binom{\ell+q}{n+s}\sum_{m=0}^{\ell-p-s}\frac{\binom{\ell-p}{m+s}\binom{\ell-q}{m}}{\binom{2 \ell}{m+n+s}^{2}}
\end{equation}
for $s=0,\ldots,\ell+q$. Now we use Proposition \ref{prop:d_r non-constant} to show that $d_{r}^{\ell}(p,q)$
equals 
\begin{equation}
\frac{(2\ell+1)}{(\ell+p+1)}\frac{(\ell-q)!(\ell+q)!}{(2\ell)!}\sum_{n=0}^{\ell+q-(r+\frac{q-p}{2})}(-1)^{\ell+q-n}\frac{(p-\ell)_{\ell+q-n}(2+2\ell-n)_{n}}{(\ell+p+2)_{\ell+q-n}n!}
\end{equation}
for $r=\frac{p-q}{2},\ldots,\ell+\frac{p+q}{2}$. It follows that
\begin{equation} c^{\ell}_{n}(p,q)=\frac{2\ell+1}{\ell+p+1}\frac{(\ell-q)!(\ell+q)!}{(2\ell)!}\frac{(p-\ell)_{\ell+q-n}}{(\ell+p+2)_{\ell+q-n}}(-1)^{\ell+q-n}\frac{(2\ell+2-n)_{n}}{n!}.
\end{equation}
This proves the theorem.
\end{proof}

We can reformulate Proposition \ref{prop:d_r non-constant} in terms of 
hypergeometric series. 

\begin{cor}\label{cor:transform1} For $N\in \bbN$, $a,b,c\in \bbN$ so that
$0\leq a \leq N$, $0\leq b \leq N$, and additionally
$a\leq b$, $N\leq a+b$ and $0\leq c\leq N-b$ we have
\begin{multline*}
\sum_{m=0}^c \frac{(-c)_m\, (b+1)_m\, (b+1)_m}{(N-a-c+1)_m\, m!\, (b-N)_m}\rFs{4}{3}{-b, N-a-b-c, N-b-m+1,N-b-m+1}{N-b-c+1, -b-m,-b-m}{1} \\
=\frac{\binom{N+1}{a}}{\binom{b}{N-a-c}\binom{N-b}{N-b-c}} 
\sum_{n=0}^c \frac{(-a)_{N-b-n}\, (-1)^{N-b-n} \, (N+2-n)_n}{(N-a+2)_{N-b-n}\, n!}
\end{multline*}
\end{cor}

The ${}_4F_3$-series in the summand is not balanced. 
Note that the case $s=0$ leads to single sums, and the ${}_4F_3$ boils down
to a terminating ${}_2F_1$ which can be summed by the Chu-Vandermonde sum, so
Corollary \ref{cor:transform1} can be viewed as an extension of
Chu-Vandermonde sum \eqref{formula:ChuVandermonde}. 

The coefficients $d_{r}^{\ell}(p,q)$ of (\ref{eq:weight4}) with $|r|\le\frac{p-q}{2}$ are independent of $r$. Corollary \ref{cor:independenceofs} in case $\ell=\ell_1+\ell_2=m_1+m_2$ 
can be stated as follows.

\begin{cor}\label{cor:transform2} For $N\in \bbN$, $a,b,c\in \bbN$ so that
$0\leq a\leq N$, $0\leq b\leq N$, $b\leq a$, $a+b\leq N$, and
$0\leq c\leq N-a-b$ we have
\begin{equation*}
\begin{split}
&\frac{\binom{N-b}{a+c}\binom{N-a}{c}}{\binom{N}{a+c}^2} 
\sum_{n=0}^b \frac{(-b)_n\, (c+a-N)_n\, (a+c+1)_n \, (a+c+1)_n}
{n!\, (c+1)_n\, (a+c-N)_n\, (a+c-N)_n} \\
&\qquad\qquad \times \rFs{4}{3}{-a, -a-c, N-a-c-n+1, N-a-c-n+1}{N-a-b-c+1, -a-c-n, -a-c-n}{1}\\
=\, & \frac{N+1}{N-a+1} \binom{N}{b}^{-1} 
\sum_{m=0}^b \frac{(-a)_m\, (N-b+m+2)_{b-m}\, (-1)^m}{(N-a+2)_m\, (b-m)!}.
\end{split}
\end{equation*}
In particular, the left hand side is independent of $c$ in the range stated. 
\end{cor}

Different proofs of Corollaries \ref{cor:transform1}, \ref{cor:transform2} 
using transformation and summation formulas for 
hypergeometric series have been communicated to us by Mizan Rahman. 

\section{Proof of the symmetry for differential operators}\label{appendix_b}

\begin{proof}[Proof of Theorem \ref{conj:first_order_do}]
In terms of $\rho(x)$ and $Z(x)$, the equations \eqref{eq:conditions.symmetry2.OpE} and \eqref{eq:conditions.symmetry3.OpE} are given by
\begin{align}
0&=Z(x)A_1(x)^*+A_1(x)Z(x),  \label{eq:symmetry_O1_Z1}\\
0&=-A_1'(x)Z(x)-\frac{\rho(x')}{\rho(x)}A_1(x)Z(x)-A_1(x)Z'(x)+A_0Z(x)-Z(x)A_0\label{eq:symmetry_O1_Z2}.
\end{align}
As a consequence of the properties of symmetry of the weight $W$, 
it suffices to verify the conditions above for all the $(n,m)$-entries with 
$n\leq m$. Here we assume that $n<m$. The case $n=m$ can be done
similarly. The first equation \eqref{eq:symmetry_O1_Z1} holds true if and only if
\begin{multline*}
Z_{n,m-1}A_1(x)_{m,m-1}+Z_{n,m}A_1(x)_{m,m}+Z_{n,m+1}A_1(x)_{m,m+1}\\
+Z_{n-1,m}A_1(x)_{n,n-1}+Z_{n,m}A_1(x)_{n,n}+Z_{n+1,m}A_1(x)_{n,n+1}=0,
\end{multline*}
for all $n\leq m$. In order to prove the expression above 
we replace the coefficients of $A_1$ and $Z$ in the left hand side and we obtain
\begin{multline*}
 -\frac{m}{2\ell} \sum_{t=0}^{m-1} c(n,m-1,t)U_{n+m-2t-1}(x)-\frac{\ell-m}{\ell} \sum_{t=0}^m c(n,m,t)\,x\,U_{n+m-2t}(x) \\
+\frac{2\ell-m}{2\ell} \sum_{t=0}^{m+1} c(n,m+1,t)U_{n+m-2t+1}(x) -\frac{n}{2\ell} \sum_{t=0}^m c(n-1,m,t)U_{n+m-2t-1}(x) \\
-\frac{\ell-n}{\ell} \sum_{t=0}^m c(n,m,t)\,x\,U_{n+m-2t}(x) +\frac{2\ell-n}{2\ell} \sum_{t=0}^m c(n+1,m,t)U_{n+m-2t+1}(x).
\end{multline*}
By using the recurrence relation $xU_r(x)=\frac12 U_{r-1}(x) + \frac12 U_{r+1}(x)$ we obtain
\begin{multline*}
\sum_{t=0}^{m-1} \left[-\frac{m}{2\ell} \, c(n,m-1,t) -\frac{n}{2\ell} \, c(n-1,m,t)
-\frac{2\ell-n-m}{2\ell} \, c(n,m,t) \right] U_{n+m-2t-1}(x) \\
\sum_{t=0}^{m} \left[\frac{2\ell-m}{2\ell} \, c(n,m+1,t)  + \frac{2\ell-n}{2\ell} \, c(n+1,m,t) 
-\frac{2\ell-n-m}{2\ell} \, c(n,m,t) \right]U_{n+m-2t+1}(x)\\
+\left[ -\frac{n}{2\ell}c(n,m+1,m)-\frac{2\ell-m}{2\ell}c(n,m+1,m+1)-\frac{2\ell-m-n}{2\ell}c(n,m,m) \right]U_{n-m-1}(x)
\end{multline*}
A simple computation shows that the coefficient of $U_{n-m-1}$ in the expression above is zero. 
Now by changing the index of summation $t$ we obtain
\begin{multline}\label{eq:symmetry_1_2}
\left[ \frac{2\ell-m}{2\ell} \, c(n,m+1,0)  + \frac{2\ell-n}{2\ell} \, c(n+1,m,0)
-\frac{2\ell-n-m}{2\ell} \, c(n,m,0)   \right]U_{n+m+1}(x)\\
+\sum_{t=0}^{m-1} \left[-\frac{m}{2\ell} \, c(n,m-1,t) -\frac{n}{2\ell} \, c(n-1,m,t)
-\frac{2\ell-n-m}{2\ell} \, c(n,m,t) \right. \\
\frac{2\ell-m}{2\ell} \, c(n,m+1,t+1)  + \frac{2\ell-n}{2\ell} \, c(n+1,m,t+1)
\left.-\frac{2\ell-n-m}{2\ell} \, c(n,m,t+1) \right]U_{n+m-2t-1}(x).
 \end{multline}
Using the explicit expression of $c(n,m,t)$ in \eqref{eq:expression_Z}
we obtain that \eqref{eq:symmetry_1_2} is given by
\begin{multline*}
\sum_{t=0}^m c(n,m,t)\left[-\frac{(m+1-t+n)(2\ell-m+1)}{2\ell(-m+1+t-n+2\ell)} \right.
-\frac{(-n+1+2\ell)(m+1-t+n)}{2\ell(-m+1+t-n+2\ell)} -\frac{2\ell-n-m}{2\ell}  \\
+ \frac{(2\ell+1-t)(m+1)}{2\ell(t+1)} + \frac{(2\ell+1-t)(n+1)}{2\ell(t+1)} 
\left. + \frac{(2\ell-n-m)(m+1-t+n)(2\ell+1-t)}{\ell(-m+1+t-n+2\ell)(t+1)} \right]U_{n+m-2t-1}(x) =0,
 \end{multline*}
since the sum of the terms in the square brackets is zero. 
This completes the proof of \eqref{eq:symmetry_O1_Z1}.

Now we prove \eqref{eq:symmetry_O1_Z2}. The $(n,m)$-entry of 
the right hand side of \eqref{eq:symmetry_O1_Z2} is given by
\begin{multline*}
x(A_1(x)Z(x))_{n,m}-(1-x^2)(A_1(x)Z'(x))_{n,m}\\
+(1-x^2)[(A_0)_{n,n}-A'_1(x)_{n,n}-(A_0)_{m,m}]Z_{n,m},
\end{multline*}
Using \eqref{eq:expression_Z} we obtain
\begin{multline}\label{eq:first_o_do_eq2_1}
(1-x^2)[(A_0)_{n,n}-A'_1(x)_{n,n}-(A_0)_{m,m}]Z_{n,m} \\
=\sum_{t=0}^{m} \frac{\ell(n-m+1)-m}{\ell} c(n,m,t)(1-x^2)U_{n+m-2t}(x).
\end{multline}
\begin{multline}\label{eq:first_o_do_eq2_2}
 x(A_1(x)Z(x))_{n,m}=\sum_{t=0}^{m} \left[ -\frac{n}{2\ell} c(n-1,m,t) \, x \,U_{n+m-2t-1}(x) \right. \\
\left. - \frac{\ell -n}{\ell} c(n,m,t) \, x^2 \,U_{n+m-2t}(x) 
+ \frac{2\ell -n}{2\ell} c(n+1,m,t) \, x \,U_{n+m-2t+1}(x) \right]
\end{multline}
\begin{multline}\label{eq:first_o_do_eq2_3}
 (1-x^2)(A_1(x)Z'(x))_{n,m}=\sum_{t=0}^{m} \left[ -\frac{n}{2\ell} c(n-1,m,t) \, (1-x^2) \,U'_{n+m-2t-1}(x) \right. \\
 - \frac{\ell -n}{\ell} c(n,m,t) \, x(1-x^2) \,U'_{n+m-2t}(x)
\left. + \frac{2\ell -n}{2\ell} c(n+1,m,t) \, (1-x^2) \,U'_{n+m-2t+1}(x) \right]
\end{multline}
Now we proceed as in the proof of the condition \eqref{eq:symmetry_O1_Z1}. 
In \eqref{eq:first_o_do_eq2_1} and \eqref{eq:first_o_do_eq2_2} we 
use the three term recurrence relation for the
Chebychev's polynomials to get rid of the factors $x$ and $x^2$. Equation 
\eqref{eq:first_o_do_eq2_3} involves the derivative of the polynomials $U$. For this we use
the following identity
$$U'_{n}(x)=\frac{(n+2)U_{n-1}(x)-nU_{n+1}(x)}{2(1-x^2)},\quad n\geq 0,\quad (U_{-1}\equiv0).$$
Finally we change the index of summation $t$ and we use the explicit expression of the coefficients
$c(n,m,t)$ to complete the proof.

The boundary condition \eqref{eq:boundary.conditions.OpE} can be easily checked.
\end{proof}

\begin{proof}[Proof of Theorem \ref{thm:first_order_do}]
 We will show that the conditions of symmetry in Theorem \ref{thm:equations_symmetry} hold true. 
The first equation  \eqref{eq:conditions.symmetry1} is satisfied because $A_2(x)$ is a scalar matrix.
Equation \eqref{eq:conditions.symmetry2} can be written in terms of $\rho(x)$ and $Z(x)$ in the following way
$$(6x-B_1(x))Z(x)+2(x^2-1)Z'(x)-Z(x)B_1(x)^*=0.$$
This can be checked by a similar computation to that of the proof of
Theorem \ref{conj:first_order_do}.

Now we give the proof of the third condition for symmetry. If we take the derivative of \eqref{eq:conditions.symmetry2},
we multiply it by 2 and we add it to \eqref{eq:conditions.symmetry3} we obtain the following equivalent condition
\begin{equation}\label{eq:eq_symmetry_combination23}
(W(x)B_1(x)^*-B_1(x)W(x))'-2(W(x)B_0-B_0W(x))=0.
\end{equation}
We shall prove instead that  
$$W(x)B_1(x)^*-B_1(x)W(x)-2\left(\int W(x) dx \right)B_0-2B_0\left(\int W(x) dx \right)=0,$$
which is obtained by integrating \eqref{eq:eq_symmetry_combination23} with respect to $x$. Then 
\eqref{eq:eq_symmetry_combination23} will follow by taking the derivative with respect to $x$.

We assume $n<m$. The other cases can be proved similarly. 
We proceed as in the proof of \eqref{eq:symmetry_O1_Z1} in Theorem \ref{conj:first_order_do} 
to show that
\begin{multline}\label{eq:third_eq_symm_first_part}
(W(x)B_1(x)^*-B_1(x)W(x))_{n,m}\\
=-\rho(x)\sum_{t=0}^{m} c(n,m,t){\frac { \left( m-n \right)  \left( \ell+1 \right)  \left( 4\,{\ell}^{2}
-\ell m-\ell n+5\ell+3 \right) }{\ell \left( -m+1+t-n+2\ell \right)  \left( t+1 \right) }}U_{n+m-2t-1}(x)
 \end{multline}
On the other hand we have
\begin{multline}\label{eq:third_cond_last}
-\left(2\left(\int W(x) dx \right)B_0+2B_0\left(\int W(x) dx \right)\right)_{n,m}\\
=\sum_{t=0}^m \left[2c(n,m,t)((B_0)_{m,m}-(B_0)_{n,n}) \int \rho(x)U_{n+m-2t}(x)dx\right].
\end{multline}
It is easy to show that the following formula for the Chebyshev's polynomials holds
$$\int \rho(x)U_i(x)=\rho(x)\left( \frac{U_{i+1}(x)}{2(i+2)}-\frac{U_{i-1}(x)}{2i}\right),\quad(U_{-1}\equiv0).$$
Therefore we have that \eqref{eq:third_cond_last} is given by
\begin{multline}\label{eq:last_SO}
\rho(x)\sum_{t=0}^m c(n,m,t) \frac{(B_0)_{m,m}-(B_0)_{n,n})}{(n+m-2t)}
\left( \frac{c(n,m,t+1)}{c(n,m,t)}-1\right)U_{n+m-2t-1}(x)\\
=\rho(x)\sum_{t=0}^{m} c(n,m,t){\frac { \left( m-n \right)  \left( \ell+1 \right)  \left( 4\,{\ell}^{2}
-\ell m-\ell n+5\ell+3 \right) }{\ell \left( -m+1+t-n+2\ell \right)  \left( t+1 \right) }}U_{n+m-2t-1}(x).
\end{multline}
Now \eqref{eq:last_SO} is exactly the negative of \eqref{eq:third_eq_symm_first_part}. This completes  the proof of the theorem.
\end{proof}

\end{document}